\documentclass[a4paper, 10pt]{amsart}
\usepackage{amsmath,amsthm}
\usepackage{amssymb, color, hyperref}
\usepackage[mathscr]{eucal}

\def\doi#1{{\small\href{https://doi.org/#1}{\path{doi:#1}}}}
\def\arxiv#1{{\small\href{http://www.arxiv.org/abs/#1}{\path{arXiv:#1}}}}
\def\url#1{{\small\href{#1}{\path{#1}}}}

\theoremstyle{plain}
\newtheorem{theorem}{\bf Theorem}[section]

\newtheorem{lemma}[theorem]{\bf Lemma}
\newtheorem{corollary}[theorem]{\bf Corollary}

\theoremstyle{definition}
\newtheorem{example}[theorem]{\bf Example}

\newcommand{\N}{\mathbb N}
\newcommand{\Z}{\mathbb Z}
\newcommand{\R}{\mathbb R}
\newcommand{\Q}{\mathbb Q}

 \DeclareMathOperator{\ord}{ord}

 \DeclareMathOperator{\supp}{supp}

\DeclareMathOperator{\Int}{Int} \DeclareMathOperator{\End}{End}
\DeclareMathOperator{\canc}{canc}

\newcommand{\DP}{\negthinspace : \negthinspace}
\newcommand{\red}{{\text{\rm red}}}

\newcommand{\eq}{\text{\rm eq}}
\newcommand{\adj}{\text{\rm adj}}

\newcommand{\BF}{\text{\rm BF}}
\newcommand{\FF}{\text{\rm FF}}

\numberwithin{equation}{section}

\subjclass[2010]{20M13, 20M14; 13A05, 16D70, 16U30}
\thanks{This work was supported by the Austrian Science Fund FWF, Project P33499-N and  by the National Natural Science Foundation of China, Grant No. 12001331.}

\begin{document}

\title{A characterization of length-factorial Krull monoids}

\author{Alfred Geroldinger and  Qinghai Zhong}

\address{University of Graz, NAWI Graz \\
Institute for Mathematics and Scientific Computing \\
Heinrichstra{\ss}e 36\\
8010 Graz, Austria}
\address{School of Mathematics and statistics, Shandong University of Technology, Zibo, Shandong 255000, China}
\email{alfred.geroldinger@uni-graz.at,  qinghai.zhong@uni-graz.at}
\urladdr{https://imsc.uni-graz.at/geroldinger, https://imsc.uni-graz.at/zhong/}

\begin{abstract}
An atomic monoid is length-factorial if each two distinct factorizations of any element have distinct factorization lengths. We provide a characterization of length-factorial Krull monoids in terms of their class groups and the distribution of prime divisors in the classes.
\end{abstract}

\maketitle


\section{Introduction and Main Results} \label{1}

By an atomic monoid, we mean a commutative  unit-cancellative semigroup with identity in which every non-invertible element is a finite product of irreducible elements. The monoids we have in mind stem from ring and module theory. An atomic monoid $H$ is said to be
\begin{itemize}
\item {\it half-factorial} if for every element $a \in H$ each two factorizations of $a$ have the same length;

\item {\it length-factorial} if for every element $a \in H$ each two distinct factorizations of $a$ have distinct lengths.
\end{itemize}
Thus, an atomic monoid is factorial if and only if it is half-factorial and length-factorial. A commutative ring is said to be atomic (half-factorial resp. length-factorial) if its monoid of regular elements has the respective property. All these arithmetical properties can be characterized in terms of catenary degrees. Indeed, it is easy to verify that a monoid is factorial (half-factorial resp. length-factorial) if its catenary degree $\mathsf c (H)=0$ (its adjacent catenary degree $\mathsf c_{\adj}(H)=0$ resp. its equal catenary degree $\mathsf c_{\eq} (H)=0$).
Half-factoriality has been studied since the beginning of factorization theory and there is a huge amount of literature.
Monotone and equal catenary degrees were first studied by Foroutan (\cite{Fo06a}), and for some recent contributions we refer to \cite{Ha09c, Ph15a, Ge-Gr-Sc-Sc10,  Ge-Yu13a, Ge-Re19d}. Length-factoriality was first studied (in different terminology) by Coykendall and Smith (\cite{Co-Sm11a}), who showed that an atomic integral domain is length-factorial if and only if it is factorial. However, such a result is far from being true in the monoid case (we refer to  recent contributions by Chapman, Coykendall, Gotti, and others \cite{C-C-G-S21, Go20a, Go20b, CM-Fo21a} as well as to  work on monoids that are not length-factorial \cite{C-G-L-M11, Ga-Ga-Ma22a}).

In the present paper we focus on Krull monoids. Krull monoids are atomic and they are factorial if and only if their class group is trivial.  Let $H$ be a Krull monoid with class group $G$ and let $G_P \subset G$ denote the set of classes containing prime divisors. Then $H$ is half-factorial if and only if the monoid of zero-sum sequences $\mathcal B (G_P)$ over $G_P$ is half-factorial. There is a standing conjecture that for every abelian group $G^*$ there is a half-factorial Krull monoid (equivalently, a half-factorial Dedekind domain) with class group isomorphic to $G^*$ (\cite[Section 5]{Gi06a}). The conjecture holds true for Warfield groups but not even for finite cyclic groups $G$ the structure or the maximal size of subsets $G_0 \subset G$, for which $\mathcal B (G_0)$ is half-factorial, are known in general (\cite{Pl-Sc05a, Pl-Sc05b}).

Our main result provides a characterization of when a Krull monoid is length-factorial, in terms of the class group and the distribution of prime divisors in the classes. Recall that reduced Krull monoids are uniquely determined by their class groups and by the distribution of prime divisors in the classes \cite[Theorem 2.5.4]{Ge-HK06a}.

\newpage
\begin{theorem} \label{1.1}
Let $H$ be a Krull monoid. Then $H = H^{\times} \times \mathcal F (P_0) \times H^*$, where $P_0$ is a set of representatives of prime elements of $H$, $\mathcal F (P_0) \times H^* \cong H_{\red}$, and $H^*$ is a reduced Krull monoid without primes. The class groups $\mathcal C (H)$ of $H$ and $\mathcal C (H^*)$ of $H^*$ are isomorphic, and $H$ is length-factorial if and only if $H^*$ is length-factorial. Let $G_{P^*} \subset \mathcal C (H^*)$ denote the set of classes containing prime divisors. \newline Then $H$ is length-factorial but not factorial if and only if every  class of $G_{P^*}$ contains precisely one prime divisor,  $H^* \cong \mathcal B (G_{P^*})$,
\[
G_{P^*}=\{e_{1,1},\ldots, e_{1,t}, e_{2,1},\ldots, e_{2,t}, \ldots, e_{k,1},\ldots, e_{k,t},g_1,\ldots, g_k,e_{0,1}, \ldots, e_{0,t}, g_0\}\,, \quad \text{and}
\]
\[
\mathcal C(H^*)=\langle e_{1,1},\ldots,e_{1,t},g_1 \rangle \oplus \ldots \oplus \langle e_{k,1},\ldots,e_{k,t},g_k \rangle\cong (\Z^t\oplus \Z/n\Z)^k \,,
\]
where
\begin{itemize} 	
\item  $t\in \N_0$, $k, s_0, s_1,\ldots,s_t\in \N$  with $k+1\neq s_0+s_1+\ldots+s_t\ge 2$,  independent elements \newline $e_{1,1},\ldots, e_{1,t}, e_{2,1},\ldots, e_{2,t}, \ldots, e_{k,1},\ldots, e_{k,t}\in \mathcal C(H^*)$ of infinite order and independent elements $g_1,\ldots, g_k\in \mathcal C(H^*)$, which are of infinite order in case $t>0$ and of finite order for $t=0$;
 	
\item $s_0$ is the smallest integer such that $s_0g_i\in \langle e_{i,1},\ldots,e_{i,t}\rangle$ and $-s_0g_i=s_1e_{i,1}+\ldots+s_te_{i,t}$ for every $i\in [1,k]$;
 		
\item $e_{0,j}=-\sum_{i=1}^ke_{i,j}$ for all $j\in [1,t]$,  $g_0=-\sum_{i=1}^kg_i$, and $n=\gcd(s_0,\ldots,s_t)$.
\end{itemize}
Moreover, $\mathcal C (H^*)$ is a torsion group if and only if $t=0$ and in that case we have $\mathcal C (H^*) \cong (\Z/n\Z)^k$, where $n\ge 2$, $k\in \N$ with $k+1\neq n$, and $\ord(g_i)=n$ for all $i\in [1,k]$.
\end{theorem}

\smallskip
Theorem \ref{1.1} shows in particular that, if $H$ is a length-factorial Krull monoid, then $H^*$ is finitely generated Krull with torsion-free quotient group, whence $H^*$ is a normal affine monoid in the sense of combinatorial commutative algebra (\cite{Br-Gu09a}).
We proceed with a series of corollaries. Based on the algebraic characterization of length-factorial Krull monoids given in Theorem \ref{1.1}, we start with the description of their arithmetic. We explicitly determine the system   $\mathcal L (H)$ of sets of lengths, which has been done only in seldom cases (\cite{Ge-Sc-Zh17b}). In particular,  the set of distances and the  elasticity  are finite (a geometric characterization of when the elasticity of Krull monoids with finitely generated class group are finite can be found in \cite{Gr22a}). Moreover, we observe that $\mathcal L (H)$ is additively closed, a   quite rare property (\cite{Ge-Sc21a}).

\smallskip
\begin{corollary}[{\bf Arithmetic of length-factorial Krull monoids}] \label{1.2}
Let $H$ be a length-factorial Krull monoid, that is not factorial,  and let all notation be as in Theorem \ref{1.1}.
\begin{enumerate}
\item The inclusion $\mathcal B (G_{P^*}) \hookrightarrow \mathcal F (G_{P^*})$ is a divisor theory with class group isomorphic to $\mathcal C (H)$. The set of atoms $\mathcal A (G_{P^*}) = \{U_0,\ldots, U_k, V_0,\ldots, V_t\}$ where, for every $i\in [0,k]$ and  every $j\in [1,t]$,
      \[
      U_0=g_0^{s_0}e_{0,1}^{s_1} \cdot \ldots \cdot e_{0,t}^{s_t}, \quad  U_i=e_{i,0}^{s_0} \cdot \ldots \cdot e_{i,t}^{s_t}, \quad V_0=g_0 \cdot \ldots \cdot g_k, \quad V_j=e_{0,j} \cdot \ldots \cdot e_{k,j} \,
      \]
      and  $U_0 \cdot \ldots \cdot U_k=V_0^{s_0} \cdot \ldots \cdot V_t^{s_t}$.

\item Every $B\in \mathcal B(G_{P^*})$ can be  written uniquely in the form
      \[
      B =(U_0 \cdot \ldots \cdot U_k)^x\ \prod_{i=0}^kU_i^{y_i} \ \prod_{j=1}^tV_j^{z_j}\,,
      \]
      where $x, y_0,\ldots, y_k, z_0,\ldots,z_t \in \N_0$,   $y_i=0$ for some  $i\in [0,k]$,  and   $z_j<s_j$ for some $j\in [0,t]$. Furthermore, we have
      \[
      \mathsf L(B) =\sum_{i=0}^ky_i + \sum_{j=0}^t z_j + \Big\{\nu (k+1)+(x-\nu)\sum_{j=0}^ts_j\colon \nu \in [0,x] \Big\}\,.
      \]
 	
\item For the system of sets of lengths $\mathcal L (H)$, we have
      \[
      \mathcal L(H)= \Big\{  \big\{y+ \nu (k+1)+(x-\nu)\sum_{j=0}^ts_j\colon \nu \in [0,x] \big\} \colon y,x\in \N_0 \Big\} \,.
      \]
      In particular, the system  $\mathcal L (H)$    is additively closed with respect to set addition as operation.
\end{enumerate}
\end{corollary}

\smallskip
Next we consider Krull monoids having some key properties, namely  the approximation property or the property that every class contains at least one prime divisor. All Krull domains have the approximation property. Holomorphy rings in global fields are Dedekind domains with finite class group and infinitely many prime divisors in all classes. Cluster algebras that are Krull (\cite{Ga-La-Sm19a}) and monoid algebras that are Krull (\cite{Fa-Wi21b}) are more recent examples of Krull domains having infinitely many prime divisors in all classes. Examples of Krull monoids stemming from module theory and having prime divisors in all classes will be discussed in Section \ref{2}.
Corollary \ref{1.3} should be compared with the classical result that a  Krull monoid having  prime divisors in each class is  half-factorial if and only if its class group has at most two elements.

\smallskip
\begin{corollary} \label{1.3}
Let $H$ be a Krull monoid and $H^*$ be as in Theorem \ref{1.1}.
\begin{enumerate}
\item If $H$ satisfies the approximation property, then $H$ is length-factorial if and only if it is factorial.

\item Suppose that every nonzero class of $H$ contains a prime divisor. Then $H$ is length-factorial  if and only if $H^* \cong \mathcal B ( \mathcal C (H) \setminus \{0\})$ and $\big( |\mathcal C (H)| \le 3$  or $\mathcal C (H)$ is an elementary $2$-group of rank two$\big)$.
\end{enumerate}
\end{corollary}

\smallskip
As already said before, it was proved by Coykendall and Smith that a commutative integral domain is length-factorial if and only if it is factorial  (\cite{Co-Sm11a}). Our next corollary shows that this result remains true for commutative Krull rings with zero divisors and for normalizing (but not necessarily commutative) Krull rings.

\smallskip
\begin{corollary}[{\bf Length-factorial Krull rings}] \label{1.4}~

\begin{enumerate}
\item Let $R$ be an additively regular Krull ring. Then $R$ is length-factorial if and only if $R$ is factorial.

\item Let $R$ be a normalizing Krull ring. Then $R$ is length-factorial if and only if $R$ is factorial.
\end{enumerate}
\end{corollary}

\smallskip
We end with a corollary on transfer Krull monoids.
A monoid $H$ is said to be {\it transfer Krull} if there is a transfer homomorphism $\theta \colon H \to B$, where $B$ is a Krull monoid. Thus, Krull monoids are transfer Krull, with $\theta$ being the identity. However, in general,  transfer Krull monoids need neither be cancellative nor completely integrally closed nor $v$-noetherian. We discuss an example after the proof of Corollary \ref{1.5} (Example \ref{2.7}) and  refer to the survey \cite{Ge-Zh20a} for more. In particular, all half-factorial monoids are transfer Krull but not necessarily Krull. But reduced length-factorial transfer Krull monoids are Krull, as we show in our final corollary.

\smallskip
\begin{corollary}[{\bf Length-factorial transfer Krull monoids}] \label{1.5}
Let $H$ be a  transfer Krull monoid. If $H$ is length-factorial, then $H_{\red}$ is Krull whence it fulfills the structural description given in Theorem \ref{1.1}.	
\end{corollary}

\smallskip
All results of the present paper, as well as prior work done in \cite{C-C-G-S21}, indicate that length-factoriality is a much more exceptional property than half-factoriality and that this is true not only for domains (which is known since \cite{Co-Sm11a}) but also for commutative and cancellative monoids. The innocent Example \ref{2.2} seems to suggest that the situation is quite different for commutative semigroups that are unit-cancellative but not necessarily cancellative.

\medskip
\section{Background on Krull monoids} \label{2}

Our notation and terminology are consistent with \cite{Ge-HK06a}.  We gather some key notions. For every positive integer $n \in \N$, $C_n$ denotes a cyclic group with $n$ elements. For integers $a, b \in \Z$, $[a, b ] = \{ x \in \Z \colon a \le x \le b \}$ denotes the discrete interval between $a$ and $b$. For subsets $A, B \subset \Z$,  $A + B = \{a+b \colon a \in A, b \in B \}$ denotes their sumset and the set of distances $\Delta (A) \subset \N$ is the set of all $d \in \N$ for which there is an element $a \in A$ such that $[a,a+d] \cap A = \{a, a+d\}$. For a set $L \subset \N$, we let $\rho (L) = \sup L/ \min L \in \Q_{\ge 1} \cup \{\infty\}$ denote the elasticity of $L$, and we set $\rho ( \{0\}) = 1$.

Let $H$ be a commutative semigroup with identity. We denote by $H^{\times}$ the group of invertible elements. We say that $H$ is reduced if $H^{\times} = \{1\}$ and we denote by $H_{\red} = \{aH^{\times} \colon a \in H \}$ the associated reduced semigroup. An element $u \in H$ is said to be cancellative if $au = bu$ implies that $a=b$ for all $a,b,u \in H$. \newpage
The semigroup $H$ is called
\begin{itemize}
\item {\it cancellative} if all elements of $H$ are cancellative;

\item {\it unit-cancellative} if $a, u \in H$ and $a=au$ implies that $u \in H^{\times}$.
\end{itemize}
Thus, every cancellative monoid is unit-cancellative.

\smallskip
\centerline{\it Throughout this paper, a monoid means a commutative and unit-cancellative semigroup with identity.}
\smallskip

For a set $P$, let $\mathcal F (P)$ be the free abelian monoid with basis $P$. An element $a \in \mathcal F (P)$ is written in the form
\[
a = \prod_{p \in P} p^{\mathsf v_p (a)} \in \mathcal F (P) \,,
\]
where $\mathsf v_p \colon \mathcal F (P) \to \N_0$ denotes the $p$-adic valuation. Then $|a| = \sum_{p \in P} \mathsf v_p (a) \in \N_0$ is the length of $a$ and $\supp (a) = \{p \in P \colon \mathsf v_p (a) > 0 \} \subset P$ is the support of $a$.
Let $H$ be a multiplicatively written  monoid. An element $u\in H$ is said to be
\begin{itemize}
\item {\it prime} if $u\notin H^{\times}$ and, for all $a,b\in H$ with $u\mid ab$, $u\nmid a$ implies $u\mid b$.

\item {\it irreducible} (or an {\it atom}) if $u\notin H^{\times}$ and, for all $a,b\in H$, $u=ab$ implies that $a\in H^{\times}$ or $b\in H^{\times}$.
\end{itemize}
We denote by $\mathcal A (H)$ the set of atoms  of $H$ and, if $H$ is cancellative, then   $\mathsf q (H)$ is the quotient group of $H$. The free abelian monoid $\mathsf Z (H) = \mathcal F (\mathcal A (H_{\red}))$ is the factorization monoid of $H$ and $\pi \colon \mathsf Z (H) \to H_{\red}$, defined by $\pi (u) = u$ for all $u \in \mathcal A (H_{\red})$, is the factorization homomorphism of $H$. For an element $a \in H$,
\begin{itemize}
\item $\mathsf Z_H (a) = \mathsf Z (a) =  \pi^{-1} (aH^{\times}) \subset \mathsf Z (H)$ is the {\it set of factorizations} of $a$, and

\item $\mathsf L_H (a) = \mathsf L (a) = \{ |z| \colon z \in \mathsf Z (a) \}$ is the {\it set of lengths} of $a$.
\end{itemize}
Note that $\mathsf L (a) = \{0\}$ if and only if  $a \in H^{\times}$. Then $H$ is  atomic (resp. factorial) if $\mathsf Z (a) \ne \emptyset$ (resp. $|\mathsf Z (a)|=1$) for all $a \in H$. Examples of atomic monoids, that are not necessarily cancellative, include semigroups of ideals and semigroups of isomorphism classes of modules (see \cite[Section 3.2 and 3.3]{F-G-K-T17}, \cite[Section 4]{Ge-Re19d}, and Examples \ref{2.2} and \ref{2.7}).
If $H$ is atomic, then
\[
1 \le |\mathsf L (a) |\le |\mathsf Z (a)| \qquad \text{for all} \quad a \in H \,,
\]
We say that the monoid $H$ is
\begin{itemize}
\item {\it half-factorial} if $1 = |\mathsf L (a)|$ for all $a \in H$, and

\item {\it length-factorial} if $1 \le |\mathsf L (a) |= |\mathsf Z (a)|$ for all $a \in H$.
\end{itemize}
Thus, by definition, $H$ is factorial if and only if it is half-factorial and length-factorial. Furthermore, $H$ is factorial (half-factorial resp. length-factorial) if and only if $H_{\red}$ has the respective property.
 Then
\[
\mathcal L (H) = \{\mathsf L (a) \colon a \in H \}
\]
is the {\it system of sets of lengths} of $H$,
\[
\Delta (H) = \bigcup_{L \in \mathcal L (H)} \Delta (L) \ \subset \N
\]
is the {\it set of distances} of $H$, and
\[
\rho (H ) = \sup \{\rho (L) \colon L \in \mathcal L (H) \} \in \R_{\ge 1} \cup \{\infty\}
\]
is the {\it elasticity} of $H$. We say that $H$ has {\it accepted elasticity} if there is $L \in \mathcal L (H)$ such that $\rho (L)=\rho (H)$.
If $H$ is not half-factorial, then $\min \Delta (H) = \gcd \Delta (H)$. We start with a simple lemma.

\smallskip
\begin{lemma} \label{2.1}
Let $H$ be a  length-factorial monoid.
\begin{enumerate}
\item $\rho (H) < \infty$.

\item If $H$ is cancellative,  then  the elasticity $\rho (H)$ is accepted.

\item If $H$ is cancellative but not factorial, then $|\Delta (H)|= 1$.
\end{enumerate}
\end{lemma}

\begin{proof}
Without restriction we may suppose that $H$ is reduced. By definition,  $H$ is half-factorial if and only if  $\rho (H)=1$ if and only if  $\Delta (H) = \emptyset$, and if this holds, then the elasticity is accepted.  Thus we may suppose that $H$ is not half-factorial.

1. Assume to the contrary that $\rho(H)$ is infinite and choose an element $a\in H$ with $\rho( \mathsf L (a))>1$. Then there exist $r,s,t\in \N_0$ and $u_1,\ldots,u_r,v_1,\ldots,v_s,w_1,\ldots,w_{t}\in \mathcal A(H)$ with $\{v_1,\ldots,v_s\}\cap \{w_1,\ldots,w_{t}\}=\emptyset$ such that
	\begin{align*}
		a=u_1 \cdot \ldots \cdot u_rv_1 \cdot \ldots \cdot v_s=u_1 \cdot \ldots \cdot u_r w_1 \cdot \ldots \cdot w_{t}\ \text{ with }\ \rho( \mathsf L (a))=(r+t)/(r+s)>1\,.
	\end{align*}
	Since $\rho(H)$ is infinite, there exists $b\in H$ such that $\rho( \mathsf L (b))>t/s$.  Moreover, there exist $r',s',t'\in \N_0$ and $$x_1,\ldots,x_{r'},y_1,\ldots,y_{s'},z_1,\ldots,z_{t'}\in \mathcal A(H)\ \text{ with }\ \{y_1,\ldots,y_{s'}\}\cap \{z_1,\ldots,z_{t'}\}=\emptyset$$ such that
	\begin{align*}
		b=x_1 \cdot \ldots \cdot x_{r'}y_1 \cdot \ldots \cdot y_{s'}=x_1 \cdot \ldots \cdot x_{r'} z_1 \cdot \ldots \cdot z_{t'}\ \text{ with }\ \rho( \mathsf L (b))=(r'+t')/(r'+s')>t/s\,.
	\end{align*}
Since
\[
\begin{aligned}
a^{t'-s'}b^{t-s} & =(u_1 \cdot \ldots \cdot u_rv_1 \cdot \ldots \cdot v_s)^{t'-s'}(x_1 \cdot \ldots \cdot  x_{r'} z_1 \cdot \ldots \cdot z_{t'})^{t-s} \\
 & =(x_1 \cdot \ldots \cdot  x_{r'}  y_1 \cdot \ldots \cdot y_{s'})^{t-s}(u_1 \cdot \ldots \cdot u_r w_1 \cdot \ldots \cdot w_{t})^{t'-s'}
\end{aligned}
\]
and since $H$ is length-factorial, we obtain that
\[
(u_1 \cdot \ldots \cdot u_rv_1 \cdot \ldots \cdot v_s)^{t'-s'}(x_1 \cdot \ldots \cdot  x_{r'} z_1 \cdot \ldots \cdot z_{t'})^{t-s} \ \text{and} \   \ (x_1 \cdot \ldots \cdot  x_{r'}  y_1 \cdot \ldots \cdot y_{s'})^{t-s}(u_1 \cdot \ldots \cdot u_r w_1 \cdot \ldots \cdot w_{t})^{t'-s'}
\]
	are equal in the factorization monoid $\mathsf Z (H)$.
Since
\[
\{v_1,\ldots,v_s\}\cap \{w_1,\ldots,w_{t}\}=\emptyset \ \text{and} \ \{y_1,\ldots,y_{s'}\}\cap \{z_1,\ldots,z_{t'}\}=\emptyset \,,
\]
it follows that	$(v_1 \cdot \ldots \cdot v_s)^{t'-s'}$ and $( y_1 \cdot \ldots \cdot y_{s'})^{t-s}$ are equal in the factorization monoid $\mathsf Z (H)$, whence $s(t'-s')=s'(t-s)$. Therefore $t/s=t'/s'>\rho( \mathsf L (b))$, a contradiction.

\smallskip
2. This proof runs along similar lines as the proof of the first assertion. But, we need to use cancellativity now which is not needed in 1. (see Example \ref{2.2}). Assume to the contrary that $\rho(H)$ is not accepted and choose an element $a\in H$ with $\rho( \mathsf L (a))>1$. Then there exist $r,s,t\in \N_0$ and $u_1,\ldots,u_r,v_1,\ldots,v_s,w_1,\ldots,w_{t}\in \mathcal A(H)$ with $\{v_1,\ldots,v_s\}\cap \{w_1,\ldots,w_{t}\}=\emptyset$ such that
	\begin{align*}
		a=u_1 \cdot \ldots \cdot u_rv_1 \cdot \ldots \cdot v_s=u_1 \cdot \ldots \cdot u_r w_1 \cdot \ldots \cdot w_{t}\ \text{ with }\ \rho( \mathsf L (a))=(r+t)/(r+s)>1\,.
	\end{align*}
	Let $a_0=v_1 \cdot \ldots \cdot v_s$. Then $\rho( \mathsf L (a_0))=t/s>1$. Since $\rho(H)$ is not accepted, there exists $b\in H$ such that $\rho( \mathsf L(b))>\rho( \mathsf L (a_0))$.  Moreover, there exist $r',s',t'\in \N_0$ and $x_1,\ldots,x_{r'},y_1,\ldots,y_{s'},z_1,\ldots,z_{t'}\in \mathcal A(H)$ with $\{y_1,\ldots,y_{s'}\}\cap \{z_1,\ldots,z_{t'}\}=\emptyset$ such that
	\begin{align*}
		b=x_1 \cdot \ldots \cdot x_{r'}y_1 \cdot \ldots \cdot y_{s'}=x_1 \cdot \ldots \cdot x_{r'} z_1 \cdot \ldots \cdot z_{t'}\ \text{ with }\ \rho( \mathsf L (b))=(r'+t')/(r'+s')>\rho(a_0)\,.
	\end{align*}
	Let $b_0=y_1 \cdot \ldots \cdot y_{s'}$. Then $\rho( \mathsf L (b_0))=t'/s'>\rho(a_0)$.
	Since $H$ is length-factorial and
	$$a_0^{t'-s'}b_0^{t-s}=(v_1 \cdot \ldots \cdot v_s)^{t'-s'}( z_1 \cdot \ldots \cdot z_{t'})^{t-s}=( y_1 \cdot \ldots \cdot y_{s'})^{t-s}(w_1 \cdot \ldots \cdot w_{t})^{t'-s'}\,,$$
	it follows from $\{v_1,\ldots,v_s\}\cap \{w_1,\ldots,w_{t}\}=\emptyset$ and $\{y_1,\ldots,y_{s'}\}\cap \{z_1,\ldots,z_{t'}\}=\emptyset$ that
	$(v_1 \cdot \ldots \cdot v_s)^{t'-s'}$ and $( y_1 \cdot \ldots \cdot y_{s'})^{t-s}$ are equal in the factorization monoid $\mathsf Z (H)$, whence $s(t'-s')=s'(t-s)$. Therefore, we infer that $\rho( \mathsf L (a_0))=t/s=t'/s'=\rho( \mathsf L (b_0))$, a contradiction.

\smallskip
3. Assume to the contrary that $|\Delta(H)|\ge 2$. Since $\min\Delta(H)=\gcd \Delta(H)$, we may choose $d,d_0\in \Delta(H)$ with $d_0\neq d$ such that $d_0$ divides $d$.
	Let $r,s,k,t\in \N_0$ and
	$$u_1,\ldots,u_r,v_1,\ldots,v_s,w_1,\ldots,w_{s+d_0}, x_1,\ldots, x_k, y_1,\ldots,y_t,z_1,\ldots,z_{t+d}\in \mathcal A(H)$$ with $\{v_1,\ldots,v_s\}\cap \{w_1,\ldots,w_{s+d_0}\}=\emptyset$ and $\{y_1,\ldots,y_t\}\cap \{z_1,\ldots,z_{t+d}\}=\emptyset$ such that
	\begin{align*}
	a=&u_1 \cdot \ldots \cdot u_rv_1 \cdot \ldots \cdot v_s=u_1 \cdot \ldots \cdot u_r w_1 \cdot \ldots \cdot w_{s+d_0}\,,&& \text{ with } \mathsf L(a)\cap [r+s, r+s+d_0]=\{r+s, r+s+d_0\}\,,\\
	b=&x_1 \cdot \ldots \cdot  x_k  y_1 \cdot \ldots \cdot y_t=x_1 \cdot \ldots \cdot  x_k z_1 \cdot \ldots \cdot z_{t+d}\,,&& \text{ with } \mathsf L(b)\cap [k+t, k+t+d]=\{k+t, k+t+d\}\,.
	\end{align*}
	Then $$a^db^{d_0}=(u_1 \cdot \ldots \cdot u_rv_1 \cdot \ldots \cdot v_s)^d(x_1 \cdot \ldots \cdot  x_k z_1 \cdot \ldots \cdot z_{t+d})^{d_0}=(x_1 \cdot \ldots \cdot  x_k  y_1 \cdot \ldots \cdot y_t)^{d_0}(u_1 \cdot \ldots \cdot u_r w_1 \cdot \ldots \cdot w_{s+d_0})^{d}\,.$$
	Since $d(r+s)+d_0(k+t+d)=d_0(k+t)+d(r+s+d_0)$ and $H$ is length-factorial, we obtain that the two factorizations  $(v_1 \cdot \ldots \cdot v_s)^d(z_1 \cdot \ldots \cdot z_{t+d})^{d_0}$ and $(y_1 \cdot \ldots \cdot y_t)^{d_0}(w_1 \cdot \ldots \cdot w_{s+d_0})^{d}$  are equal (in the factorization monoid $\mathsf Z (H)$). Since $\{v_1,\ldots,v_s\}\cap \{w_1,\ldots,w_{s+d_0}\}=\emptyset$, we obtain $(v_1 \cdot \ldots \cdot v_s)^d$ divides $(y_1 \cdot \ldots \cdot y_t)^{d_0}$ in $\mathsf Z (H)$. Since $\{y_1,\ldots,y_t\}\cap \{z_1,\ldots,z_{t+d}\}=\emptyset$, we obtain $(y_1 \cdot \ldots \cdot y_t)^{d_0}$ divides $(v_1 \cdot \ldots \cdot v_s)^d$ in $\mathsf Z (H)$, whence $(v_1 \cdot \ldots \cdot v_s)^d=(y_1 \cdot \ldots \cdot y_t)^{d_0} \in \mathsf Z (H)$. It follows that $y_1 \cdot \ldots \cdot y_t=(v_1 \cdot \ldots \cdot v_s)^{d/d_0}$ and hence
	\[
	b=x_1 \cdot \ldots \cdot x_k(v_1 \cdot \ldots \cdot v_s)^{d/d_0}=x_1 \cdot \ldots \cdot x_k(v_1 \cdot \ldots \cdot v_s)^{d/d_0-1}w_1 \cdot \ldots \cdot w_{s+d_0}\,,
	\]
	which implies that $k+t+d_0\in \mathsf L(b)\cap [k+t, k+t+d]$, a contradiction.
\end{proof}

Our next example shows that the elasticity of a non-cancellative length-factorial monoid does not need to be accepted and that the set of distances may contain more than one element.

\smallskip
\begin{example} \label{2.2}~

1. Let $R$ be a ring and $\mathcal C$ be a small class of left $R$-modules that is closed under finite direct sums, direct summands, and isomorphisms. Then the set $\mathcal V (\mathcal C)$ of isomorphism classes of modules from $\mathcal C$ is a reduced commutative semigroup, with operation induced by the direct sum (\cite{Ba-Wi13a}). Suppose that all modules from $\mathcal C$ are directly finite (or Dedekind finite), which means that
\[
\text{If $M,N$ are modules from $\mathcal C$ such that $M \cong M \oplus N$, then $N = 0$.}
\]
This property holds true for large classes of modules (including all finitely generated modules over commutative rings; for more see \cite{Go79a, Fa19a}) and is equivalent to $\mathcal V (\mathcal C)$ being unit-cancellative. We will meet such monoids $\mathcal V (\mathcal C)$ at several places of the manuscript (e.g., in Example \ref{2.7}).

\smallskip
2. For $m \in \N$, let us consider the commutative monoid $H_m$ generated by $A_m = \{a_1,\ldots,a_m,u_1,u_2\}$ with relations generated by $R_m = \{ (a_1u_1^2, a_1u_2^3), (a_2u_1^4, a_2u_2^6),\ldots, (a_mu_1^{2m},a_mu_2^{3m})\}$, say
\[
H_m = \langle a_1,\ldots,a_m,u_1,u_2 \mid a_1 u_1^2=a_1 u_2^3, a_2u_1^4=a_2u_2^6, \ldots,a_mu_1^{2m}=a_mu_2^{3m} \rangle \,.
\]
Then $H_m$ is a reduced, commutative, atomic,  non-cancellative  monoid with $\mathcal A (H_m) =A_m$. By construction, we have  $[1,m]\subset \Delta(H_m)$, $\rho (H_m) = 3/2$, and $\rho (H)$ is not accepted. We assert that $H_m$ is length-factorial.

 We define, for any $a,b\in H_m$, that $a\sim b$ if there exists $c\in H_m$ such that $ac=bc$. This is a congruence relation on $H_m$ and the monoid $H_{m, \canc}=H_m/\sim$ is the associated  cancellative monoid  of $H_m$.
 For every $a\in H_m$, we denote by $[a]\in H_{m, \canc}$ the congruence class of $H$. Then $$H_{m, \canc}\cong \mathcal F(\{[a_i]\colon i\in[1,m]\})\times \langle [u_1],[u_2]\mid [u_1]^2=[u_2]^3\rangle\,,$$
whence it is easy to see that  $H_{m, \canc}$ is length-factorial.
 Let $x_1,x_2$ be two atoms of $H_m$. By our construction of $H$, we have $[x_1]=[x_2]$ if and only if $x_1=x_2$. Therefore the length-factoriality of $H_{m, \canc}$ implies that $H_m$ is length-factorial.
By a result of Bergman-Dicks (\cite[Theorems 6.2 and 6.4]{Be74a} and \cite[page 315]{Be-Di78}), the  monoid $H_m$ can be realized as a monoid of isomorphism classes of modules, as introduced in 1.
\end{example}

Next we discuss Krull monoids. A monoid homomorphism $\varphi \colon H \to D$ is called a
\begin{itemize}
\item {\it divisor homomorphism} if $a, b \in H$ and $\varphi (a) \mid \varphi (b)$ (in $D$) imply that $a \mid b$ (in $H$);

\item {\it divisor theory} (for $H$) if $\varphi $ is a divisor homomorphism, $D$ is free abelian, and for every $a \in D$ there are $a_1, \ldots, a_m \in H$ such that $a = \gcd \big( \varphi (a_1), \ldots, \varphi (a_m) \big)$.
\end{itemize}
A monoid $H$ is a {\it Krull monoid} if it is cancellative and satisfies one of the following equivalent conditions (\cite[Theorem 2.4.8]{Ge-HK06a} ):
\begin{itemize}
\item[(a)] $H$ is completely integrally closed and satisfies the ACC on divisorial ideals.

\item[(b)] $H$ has a divisor homomorphism to a free abelian monoid.

\item[(c)] $H$ has a divisor theory.
\end{itemize}
Property (a) can be used to show that a domain is a Krull domain if and only if its multiplicative monoid of nonzero elements is a Krull monoid. Examples of Krull monoids are given in \cite{Ge-HK06a} and in the recent survey \cite{Ge-Zh20a}). In particular, let $\mathcal V (\mathcal C)$ be a monoid of isomorphism classes of modules, as introduced in Example \ref{2.2}.1. If $\End_R (M)$ is semilocal for all $M$ from $\mathcal C$, then $\mathcal V (\mathcal C)$ is a reduced Krull monoid (\cite[Theorem 3.4]{Fa02}), and
every reduced Krull monoid  can be realized as  a monoid of isomorphism classes of modules (\cite[Theorem 2.1]{Fa-Wi04}).

To discuss class groups of Krull monoids, let $H$ be a Krull monoid. Then there is a divisor theory $H_{\red} \hookrightarrow F = \mathcal F (P)$ and
\begin{equation} \label{classgroup}
\mathcal C (H) = \mathcal C (H_{\red}) = \mathsf q (F)/ \mathsf q (H_{\red})
\end{equation}
is the (divisor) class group of $H$. The divisor class group is isomorphic to the (ideal theoretic) $v$-class group of $H$, and if $R$ is a Krull domain, then the class group of the Krull monoid $R \setminus \{0\}$ coincides with the usual divisor class group of the domain $R$. If the monoid $H$  in Theorem \ref{1.1} is length-factorial, then $H^*$ is a reduced finitely generated Krull monoid. There are various characterizations of finitely generated Krull monoids (\cite[Theorem 2.7.14]{Ge-HK06a}). In particular, every such monoid is a Diophantine monoid (the monoid of non-negative solutions of a system of linear Diophantine equations; \cite{Ch-Kr-Oe02}). For every $a \in \mathsf q (F)$, we denote by $[a] = a \mathsf q (H_{\red}) \subset \mathsf q (F)$ the class containing $a$. For $g \in \mathcal C (H)$, $P \cap g$ is the set of prime divisors lying in $g$. Concerning the distribution of prime divisors in Krull monoids of isomorphism classes of modules we refer to \cite{F-H-K-W06, Fa06a, H-K-K-W06, Fa12a, Ba-Ge14b}.

Let $G$ be an additive abelian group and $G_0 \subset G$ be a subset. We denote by $\langle G_0 \rangle \subset G$ the subgroup generated by $G_0$ and by $[G_0 ] \subset G$ the submonoid generated by $G_0$. A tuple $(e_1, \ldots, e_r) \in G^r$, with $r \in \N$ (respectively, the elements $e_1, \ldots, e_r \in G$) are called {\it independent} if $e_i \ne 0$ for all $i \in [1,r]$ and $\langle e_1, \ldots, e_r \rangle = \langle e_1 \rangle \oplus \ldots \oplus \langle e_r \rangle$, and it is called a {\it basis} of $G$ if $e_i \ne 0$ for all $i \in [1,r]$ and $G = \langle e_1 \rangle \oplus \ldots \oplus \langle e_r \rangle$.

We discuss a class of Krull monoids needed in the sequel, namely monoids of zero-sum sequences.
For an element
\[
S = g_1 \cdot \ldots \cdot g_{\ell} = \prod_{g \in G_0} g^{\mathsf v_g (S)} \in \mathcal F (G_0) \,,
\]
where $g_1, \ldots, g_{\ell} \in G_0$, $|S|=\ell = \sum_{g \in G_0} \mathsf v_g (S) \in \N_0$ is the length of $S$, and
\[
\sigma (S) = g_1 + \ldots + g_{\ell} \in G \quad \text{is the {\it sum} of $S$}\,.
\]
 We say that $S$ is {\it zero-sum free} if $\sum_{i \in I} g_i \ne 0$ for all $\emptyset \ne I \subset [1, \ell]$. The {\it monoid of zero-sum sequences}
\[
\mathcal B (G_0) = \{ S \in \mathcal F (G_0) \colon \sigma (S) = 0 \} \subset \mathcal F (G_0)
\]
over $G_0$ is a Krull monoid, by Property (b),  since the inclusion $\mathcal B (G_0) \hookrightarrow \mathcal F (G_0)$ is a divisor homomorphism. We denote by $\mathcal A (G_0):= \mathcal A \big( \mathcal B (G_0) \big)$ the set of atoms (minimal zero-sum sequences) of $\mathcal B (G_0)$. The subset $G_0$ is called half-factorial (non-half-factorial resp. minimal non-half-factorial) if the monoid $\mathcal B (G_0)$ is half-factorial (not half-factorial resp. $G_0$ is not half-factorial but every proper subset is half-factorial).  Half-factorial and (minimal) non-half-factorial subsets play a central role when studying the arithmetic of Krull monoids (we refer to \cite[Chapter 6]{Ge-HK06a} for the basics and to \cite{Zh18a,Pl-Sc20a}). Note that minimal non-half-factorial subsets are finite.

The arithmetic of Krull monoids is studied via transfer homomorphisms to monoids of zero-sum sequences. We recall the required concepts. A monoid homomorphism  $\theta \colon H \to B$ is called a {\it transfer homomorphism} if it has the following properties:

\begin{enumerate}
\item[{\bf (T\,1)\,}] $B = \theta(H) B^\times$ \ and \ $\theta ^{-1} (B^\times) = H^\times$.

\item[{\bf (T\,2)\,}] If $u \in H$, \ $b,\,c \in B$ \ and \ $\theta (u) = bc$, then there exist \ $v,\,w \in H$ \ such that \ $u = vw$, \ $\theta (v) \in bB^{\times}$, and  $\theta (w) \in c B^{\times}$.
\end{enumerate}

\smallskip
\begin{lemma} \label{2.3}
Let $\theta \colon H \to B$ be a transfer homomorphism of atomic monoids.
\begin{enumerate}
\item For every $a \in H$, we have $\mathsf L_H (a) = \mathsf L_B \big( \theta (a) \big)$.

\item Let $p \in H$. Then $p$ is an atom  in $H$ if and only if $\theta (p)$ is an atom  in $B$. Moreover, if $p$ is a prime in $H$, then $\theta (p)$ is a prime in $B$.

\item $\mathcal L (H) = \mathcal L (B)$, whence $H$ is half-factorial if and only if $B$ is half-factorial.

\item If $H$ is length-factorial, then $B$ is length-factorial.
\end{enumerate}
\end{lemma}

\begin{proof}
Without restriction we may suppose that $H$ and $B$ are reduced. Then {\bf (T1)} implies that $\theta$ is surjective.

1. This easily follows from {\bf (T\,2)\,} (for details in the cancellative setting we refer to \cite[Chapter 3.2]{Ge-HK06a}).

2. Let $p \in H$. Since $p$ is an atom in $H$ if and only if $\mathsf L_H (p) = \{1\}$ and similarly for $\theta (p)$ and $B$, 1. implies that $p$ is an atom in $H$ if and only if $\theta (p)$ is an atom in $B$.

Now suppose that $p$ is a prime in $H$ and let $\alpha, \beta \in B$ such that $\theta (p) \mid \alpha \beta$. Then there is $c \in H$ such that $\alpha \beta = \theta (pc)$. Then {\bf (T2)} implies that there are $a, b \in H$ such that $pc = ab$, $\theta (a) = \alpha$, and $\theta (b) = \beta$. Without restriction we may suppose that $p \mid a$, say $a = pa'$ for some $a' \in H$, whence $\alpha = \theta (a) = \theta (p) \theta (a')$. Thus $\theta (p)$ is a prime in $B$.

3. This follows immediately from 1.

4. Suppose that $H$ is length-factorial and choose some $\alpha \in B$. Let $a \in H$ such that $\theta (a) = \alpha$, and let $k \in \mathsf L_B ( \alpha ) = \mathsf L_H (a)$. By {\bf (T2)}, every factorization of $\alpha$ of length $k$ can be lifted to a factorization of $a$ of length $k$. Thus, if there is only one factorization of $a$ of length $k$, there is only one factorization of $\alpha$ of length $k$. This implies that $B$ is length-factorial.
\end{proof}

Let all notation be as in Lemma \ref{2.3}. There are examples (even for cancellative monoids) where $\theta (p)$ is a prime in $B$ but $p$ fails to be prime in $H$. Furthermore, $B$ may be length-factorial, but $H$ is not length-factorial.

\smallskip
The study of factorial versus length-factorial monoids can be seen as part of a larger program. We briefly outline this and introduce (as suggested by the reviewer) the concept of length-FF-monoids.
A monoid $H$ is said to be
\begin{itemize}
\item an {\it \FF-monoid} (finite factorization monoid) if $\mathsf Z (a)$ is finite nonempty for all $a \in H$.

\item a {\it \BF-monoid} (bounded factorization monoid) if $\mathsf L (a)$ is finite nonempty for all $a \in H$.

\item a {\it length-\FF-monoid} if it is atomic and every element has only finitely many factorizations of the same length.
\end{itemize}
A commutative ring $R$ has one of these properties if the respective property holds true for its monoid of regular elements. By definition,  a monoid is  an FF-monoid if and only if it is a BF-monoid and a length-FF-monoid.

Every Krull monoid is an \FF-monoid.
Let $H$ be a cancellative monoid. If $H$ satisfies the ACC on divisorial ideals, then $H$ is a \BF-monoid and every \BF-monoid satisfies the ACC on principal ideals. Suppose that $H$ satisfies the ACC on divisorial ideals and $(H \DP \widehat H) \ne \emptyset$. Then $\widehat H$ is a Krull monoid, and $H$ is an \FF-monoid if and only if the factor group $\widehat H^{\times}/H^{\times}$ is finite (\cite[Theorem 1.5.6]{Ge-HK06a}). In particular, a noetherian domain $R$, whose integral closure $\overline R$ is a finitely generated $R$-module, is an \FF-domain if and only if $\overline R^{\times}/R^{\times}$ is finite.
The ring of integer-valued polynomials $\Int (\Z)$ is an \FF-domain and hence a \BF-domain but it does not satisfy the ACC on divisorial ideals. We continue with two examples.

\smallskip
\begin{example}~ \label{2.4}

1. Let $H \subset (\Q_{\ge 0}, +)$ be the additive submonoid of the non-negative rationals that is generated by $\{ 1/p \colon p \ \text{is prime} \}$. Then $H$ satisfies the ACC on principal ideals by \cite[Theorem 4.5]{Ch-Go-Go21}. Since  $\mathcal A (H) = \{ {1/p} \colon p \in \mathbb P\}$ and $1 = {1/p} + \ldots + {1/p}$, it follows that $\mathbb P \subset \mathsf L_H (1)$, whence $H$ is not a \BF-monoid. We assert that $H$ is a length-FF-monoid. In order to show this, let $r=\frac{n}{m}\in H$ and let $k\in \mathsf L_H (r)$, where $m,n\in \N$ such that $\gcd(m,n)=1$. It suffices to show that there are only finitely many primes $p$ such that $1/p$ can appear in a factorization of $r$ of length $k$.
Suppose
\[
r=\frac{n}{m}=\sum_{i=1}^{t}\frac{q_i}{p_i}, \text{ where $p_1,\ldots,p_t$ are pairwise distinct primes, $q_1, \ldots, q_t \in \N$  and }\sum_{i=1}^tq_i=k\,.
\]
If $i \in [1, t]$ and $p_i$ does not divide $m$, then $p_i$ must divide $q_i$, whence $q_i \in [p_i, k]$. Together with the fact that $m$ has only finitely many prime divisors, the assertion follows.

\smallskip
2. Let $H$ be as in 1. and consider the monoid algebra
\[
\Q [H] = \Big\{ \sum_{p \in \mathbb P} r_p X^{1/p} \colon \quad  r_p \in \Q \quad \text{and} \quad r_p=0 \quad \text{for almost all} \ p \in \mathbb P \Big\} \,.
\]
Since  $\Q [H]$ satisfies the ACC on principal ideals (\cite[Proposition 4.2]{Go21a}), it is atomic. Since $X^{1/p}$ is an atom for all $p \in \mathbb P$ and $X = X^{1/p} \cdot \ldots \cdot X^{1/p}$, it follows that $\mathbb P \subset \mathsf L_{\Q[H]} (X)$, whence $\Q [H]$ is not a \BF-monoid.
We assert that $\Q[H]$ is not a length-FF-monoid. In order to prove this, we introduce some further notation.
For an element  $f=r_0X^{\alpha_0}+r_1X^{\alpha_1}+\ldots+r_kX^{\alpha_k}\in \Q[H]$	, where $k \in \N_0$, $r_0,\ldots, r_k\in \Q$, and $\alpha_0,\ldots,\alpha_k\in H$ with $0=\alpha_0<\alpha_1<\ldots<\alpha_k$,  we define  $\mathsf n(f)=r_0$, and $\mathsf m(f)=\alpha_k$. Thus $\mathsf m(f)=0$ if and only if $f\in \Q$.

For an odd prime $p$, we  consider the factorization
\[
1-X=1-(X^{1/p})^p=(1-X^{1/p})(1+X^{1/p}+\ldots+X^{(p-1)/p}) \,,
\]
and we assert that
 both, $1-X^{1/p}$ and $1+X^{1/p}+\ldots+X^{(p-1)/p}$, are atoms of $\Q[H]$.
If this holds, then $1-X$ has infinitely many factorizations of length two, whence $\Q[H]$ is not a length-FF-monoid.

Suppose that $1-X^{1/p}=gh$ for some $g,h\in \Q[H]$.   Then $1/p=\mathsf m(g)+\mathsf m(h)$ and, since $1/p$ is an atom of $H$, it follows that $\mathsf m(g)=0$ or $\mathsf m(h)=0$. Thus,  $g\in \Q$ or $h\in \Q$ and $1-X^{1/p}$	 is an atom.

Suppose that  $1+X^{1/p}+\ldots+X^{p-1/p}=gh$ for some $g,h\in \Q[H]$.
Without loss of generality, we may assume that $\mathsf n(g)=\mathsf n(h)=1$ and we set
\[
g=1+u_1X^{\alpha_1}+\ldots+u_kX^{\alpha_k} \quad \text{ and }\quad h=1+v_1X^{\beta_1}+\ldots+v_{\ell}X^{\beta_{\ell}}\,,
\]
where $k,\ell \in \N_0$, $u_0=v_0=1$, $u_1,\ldots, u_k,v_1,\ldots,v_{\ell}\in \Q\setminus\{0\}$, and $\alpha_0,\ldots, \alpha_k, \beta_0,\ldots,\beta_{\ell}\in H$ with $0= \alpha_0 <\alpha_1<\ldots< \alpha_k$ and $0 = \beta_0 <\beta_1<\ldots<\beta_{\ell}$. We have to show that either $g=1$ or $h=1$.

 The coefficient of $X^{\alpha_k+\beta_{\ell}}$ in $gh$ is $u_kv_{\ell}\neq 0$, whence $\alpha_k+\beta_{\ell}=(p-1)/p$. Since  for every $i\in [1,p-1]$, the element $i/p \in H$ has unique factorization in $H$, it follows that $\alpha_k=a/p$ and $\beta_{\ell}=b/p$ for some $a,b\in [0,p-1]$ with $a+b=p-1$.
Let \begin{align*}
	I&= \big\{i\in [1,k]\colon \alpha_i\not\in \{r/p\colon r\in [1, a-1]\} \big\}\,, \quad g_2=\sum_{i\in I}u_iX^{\alpha_i}\,, \quad  g_1=g-g_2\,,\\
	\text{and }	J&= \big\{j\in [1,\ell]\colon \beta_j\not\in \{r/p\colon r\in [1, b-1]\} \big\}\,, \quad h_2=\sum_{j\in J}v_jX_{\beta_j}\,, \quad 	h_1=h-h_2\,,
\end{align*}
with the convention that $g_2=0$ if $I = \emptyset$ and $h_2=0$ if $J = \emptyset$.
Then $gh-g_1h_1=g_1h_2+h_1g_2+g_2h_2\in \Q[X^{1/p}]$.
As above we use that $i/p$ has unique factorization in $H$ for every $i \in [1,p-1]$ and infer that  the coefficient of $X^{i/p}$ in $g_1h_2+h_1g_2+g_2h_2$ equals zero. Therefore, we obtain that $g_1h_2+h_1g_2+g_2h_2=0$, whence $gh=g_1h_1$. Since $g_1, h_1 \in \Q[X^{1/p}]$, $\mathsf m (g) = \mathsf m (g_1)$, and $\mathsf m (h) = \mathsf m (h_1)$, we set
\[
g_1=1+u_1'X^{1/p}+\ldots+u_{a}'X^{a/p}, \quad h_1=1+v_1'X^{1/p}+\ldots+v_{b}'X^{b/p}\,,
\]
where  $u_0'=v_0'=1$, $u_1',\ldots, u_{a}', v_1',\ldots,v_{b}'\in \Q$ with $u_{a}'\neq 0$, and $v_{b}'\neq 0$.
Setting $Y=X^{1/p}$ we obtain
\[
1+Y+\ldots+Y^{p-1}=(1+u_1'Y+\ldots+u_{a}'Y^{a})(1+v_1'Y+\ldots+v_{b}'Y^{b})\in \Q[Y]\,.
\]
Since $1+Y+\ldots+Y^{p-1}$ is irreducible in $\Q[Y]$,
we obtain,  after renumbering if necessary, that $b=0$, $\mathsf m (h) = 0$, and $h=1$.
\end{example}

\smallskip
The examples show that non-BF-monoids may or may not be length-FF-monoids. Thus, the length-FF-property could be a tool leading to a better understanding of the non-BF-property. Indeed, although the concepts of \BF- and \FF-monoids and domains were introduced  more than thirty years ago (\cite{An-An-Za90}), the arithmetic of non-BF-monoids has not been studied yet  in a systematic way (a main obstacle is that they miss the ACC on divisorial ideals).

Furthermore, the examples show that the length-FF-property does not always imply the BF-property, whence it need not imply the FF-property. This is in analogy to the fact that length-factoriality does not always imply factoriality. But, since the latter implication does hold true for large classes of monoids including all domains, it is a  natural question in the same vein  to ask in which classes of monoids or domains the length-FF-property implies the FF-property.

\medskip
\section{Proof of Theorem \ref{1.1} and of its corollaries} \label{3}

In this section we prove our main results. To do so, we start with three lemmas on Krull monoids.

\smallskip
\begin{lemma} \label{3.1}
Let $H$ be a reduced Krull monoid with divisor theory $H \hookrightarrow F = \mathcal F (P)$, class group $G = \mathcal C (H)$, and let $G_P = \{ [p] \colon p \in P\} \subset G$ be the set of classes containing prime divisors.
\begin{enumerate}
\item The map $\boldsymbol \beta \colon H \to \mathcal B (G_P)$, defined by $a = p_1 \cdot \ldots \cdot p_{\ell} \mapsto [p_1] \cdot \ldots \cdot [p_{\ell}]$ where $\ell \in \N$ and $p_1, \ldots, p_{\ell} \in P$, is a transfer homomorphism.

\item The map $\boldsymbol \beta$ is an isomorphism if and only if every class $g \in G_P$ contains precisely one prime divisor.

\item We have $G = [G_P]$ and $G = [G_P \setminus \{g\}]$ for all classes $g \in G_P$ that contain precisely one prime divisor.
\end{enumerate}
\end{lemma}

\begin{proof}
1. This follows from  \cite[Theorem 3.4.10]{Ge-HK06a}.

2. Since $\mathcal B (G_P)$ is reduced, {\bf (T\,1)\,} implies that $\boldsymbol \beta$ is surjective. Thus, $\boldsymbol \beta$ is an isomorphism if and only if every class $g \in G_P$ contains precisely one prime divisor.

3. This follows from \cite[Theorem 2.5.4]{Ge-HK06a}.
\end{proof}

\smallskip
\begin{lemma}\label{3.2}
Let $G$ be an abelian group and let $G_0\subset G\setminus\{0\}$ be a subset such that $G = [G_0\setminus\{g\}]$ for all $g \in G_0$. Suppose there is $B\in \mathcal B(G_0)$ having two distinct factorizations
	$$B=U_1\cdot \ldots \cdot U_k=V_1\cdot \ldots \cdot V_{\ell}\,,$$
	where $k,\ell\ge 2$ and
 $U_1,\ldots, U_k, V_1,\ldots, V_{\ell}\in \mathcal A(G_0)$.
\begin{enumerate}
\item For any distinct $g,h\in G_0$, there exists two atoms $A_1, A_2\in \mathcal A(G_0)$ such that $\mathsf v_g(A_1)=1$ and $h\in \supp(A_2)\subset G_0\setminus\{g\}$.
 	
\item If $\mathcal B(G_0)$ is  length-factorial, then $\mathcal A(G_0)=\{U_1,\ldots, U_k, V_1,\ldots, V_{\ell}\}$.
\end{enumerate}
\end{lemma}

\begin{proof}
1. Let $g,h\in G_0$ with $g\neq h$. Since $-h\in G=[G_0\setminus \{g\}]$, there is an atom $A_2\in \mathcal A(G_0\setminus \{g\})$ such that $h\in \supp(A_2) \subset G_0\setminus\{g\}$. Since $-g\in G=[G_0\setminus\{g\}]$, there is an atom $A_1\in \mathcal A(G_0)$ such that $\mathsf v_g(A_1)=1$.

2. Suppose $\mathcal B(G_0)$ is  length-factorial. Assume to the contrary there is an atom $A\in \mathcal A(G_0)\setminus\{U_1,\ldots, U_k, V_1,\ldots, V_{\ell}\}$.
If $|\supp(A)|=1$, say $\supp(A)=\{g\}$, then $\ord(g)$ is finite and  by 1. there exists an atom $A_1$ with $\mathsf v_g(A_1)=1$,  whence $A_1\neq A$. Therefore $A$ divides $A_1^{\ord(g)}$.
If $|\supp(A)|\ge 2$, then for every $g\in \supp(A)$, it follows by 1. that there exists an atom $A_g\in \mathcal A(G_0)$ with $g\in \supp(A_g)$ such that $\supp(A)\not\subset \supp(A_g)$. Then $A\neq A_g$ for every $g\in \supp(A)$ and $A$ divides $\prod_{g\in \supp(A)}A_g^{\mathsf v_g(A)}$.

To sum up, there exist  $s\in \N$ and atoms $W_1, \ldots, W_s$ with $A\neq W_i$ for every $i\in [1,s]$ such that
 $A$ divides $W_1\cdot \ldots \cdot W_s$.  We may suppose 	
$W_1\cdot \ldots \cdot W_{s}=AX_2\cdot \ldots \cdot X_t$, where $t\ge 2$ and $X_2,\ldots, X_t\in \mathcal A(G_0)$. If $\ell=k$ or $t=s$, then $\mathcal B(G_0)$ is not length-factorial, a contradiction. Suppose $\ell\neq k$ and $t\neq s$.
By symmetry, we may suppose that  $\ell> k$.
If $t> s$, then
\[
(W_1\cdot \ldots \cdot W_{s})^{\ell-k} (V_1\cdot \ldots \cdot V_{\ell})^{t-s}= (U_1\cdot \ldots \cdot U_k)^{t-s} (AX_2\cdot \ldots \cdot X_t)^{\ell-k}
\]
has two distinct factorizations of length $\ell t-sk$, whence $\mathcal B(G_0)$ is not length-factorial, a contradiction.
If $s> t$, then
\[
(W_1\cdot \ldots \cdot W_{s})^{\ell-k} (U_1\cdot \ldots \cdot U_k)^{s-t}= (V_1\cdot \ldots \cdot V_{\ell})^{s-t} (AX_2\cdot \ldots \cdot X_t)^{\ell-k}
\]
has two distinct factorizations of length $s\ell -tk$, whence $\mathcal B(G_0)$ is not length-factorial, a contradiction.
\end{proof}

\smallskip
\begin{lemma}\label{3.3}
Let $G$ be an abelian group and let $G_0\subset G\setminus\{0\}$ be a subset such that $[G_0\setminus\{g\}]=G$ for all $g\in G_0$. Suppose that  $\mathcal B(G_0)$ is length-factorial but not factorial.
\begin{enumerate}
\item $G_0$ is a minimal non-half-factorial set.

\item For every $g\in G_0$, there exist  $A_1\in \mathcal A(G_0)$ such that $\mathsf v_g(A_1)=1$ and $|\{A\in \mathcal A(G_0)\colon \mathsf v_g(A)>0\}|=2$.

\item For any two distinct atoms $A_1, A_2\in \mathcal A(G_0)$,  either $\supp(A_1)\cap \supp(A_2)=\emptyset$ or $|\gcd(A_1, A_2)|=1$.
\end{enumerate}	
\end{lemma}

\begin{proof}
Since $\mathcal B (G_0)$ is length-factorial but not factorial, it is not half-factorial.

	1. There is a $B_0 \in \mathcal B(G_0)$ such that $|\mathsf L(B_0)|\ge 2$, which implies that $\supp(B_0)$ is not half-factorial. Let $G_1\subset \supp(B_0)$ be a minimal non-half-factorial subset and let $B_1 \in \mathcal B(G_1)$ such that $|\mathsf L(B_1)|\ge 2$. Then Lemma \ref{3.2}.2 implies $\mathcal A(G_0)=\mathcal A(G_1)$.

Assume to the contrary that $G_0\setminus G_1\neq \emptyset$.
 Let $h\in G_0\setminus G_1$. Then by Lemma \ref{3.2}.1 there is an atom $A\in \mathcal A(G_0)$ with $h\in \supp(A)$, whence $A\not\in \mathcal A(G_1)$, a contradiction. Therefore $G_0=G_1$ is  a minimal non-half-factorial subset.

2. Let $g\in G_0$. By Lemma \ref{3.2}.1, there exists an atom $A_1$ such that $\mathsf v_g(A_1)=1$ and hence $|\supp(A_1)|\ge 2$.
Let $h_0\in \supp(A_1)\setminus\{g\}$. Then Lemma \ref{3.2}.1 implies  there exists an atom  $A_g\in \mathcal B(G_0\setminus\{h_0\})$ such that $g\in \supp(A_g)$. Thus $A_g\neq A_1$.
Furthermore, for every $h\in \supp(A_1)\setminus\{g\}$,  Lemma \ref{3.2}.1 implies that there exists an atom  $A_h\in \mathcal B(G_0\setminus\{g\})$ such that $h\in \supp(A_h)$.

Assume to the contrary that there exists an atom $A_3\in \mathcal A(G_0)\setminus\{A_1,A_g\}$ such that $g\in\supp(A_3)$.
Therefore $$A_3\prod_{h\in G_0\setminus \{g\}}A_h^{\mathsf v_g(A_3)\mathsf v_h(A_1)}=A_1^{\mathsf v_g(A_3)} X_1\cdot \ldots \cdot X_{s}\,,$$
where $s\in \N$ and $X_1,\ldots,X_s\in \mathcal A(G_0\setminus \{g\})$.
It follows by Lemma \ref{3.2}.2 that
$A_g\in \{A_1,A_3\}\cup \{A_h\colon h\in \supp(A)\setminus\{g\}\}\cup \{X_i\colon i\in [1,s]\}$, a contradiction.

\medskip
3. Let $A_1, A_2\in \mathcal A(G_0)$ be  distinct  such that $\supp(A_1)\cap \supp(A_2)\neq\emptyset$. Assume to the contrary that there are $g,h\in G_0$ such that $gh$ divides  $\gcd(A_1, A_2)$.
 By 2., there is no other atom $A$ such that $\supp(A)\cap \{g,h\}\neq \emptyset$. If $g=h$, then there is no atom $A$ with $\mathsf v_g(A)=1$, a contradiction to 2. If $g\neq h$, then $-h\in [G_0\setminus\{g\}]$ implies that there is an atom $A\in \mathcal A(G_0\setminus \{g\})$ with $h\in \supp(A)$, a contradiction.
\end{proof}

\smallskip
\begin{proof}[Proof of Theorem \ref{1.1}]
Let $H$ be a Krull monoid. By \cite[Theorem 2.4.8]{Ge-HK06a}, there is a decomposition $H = H^{\times} \times H_0$, where $H_0$ is a reduced Krull monoid, isomorphic to $H_{\red}$. If $P_0 \subset H_0$ is the set of prime elements of $H_0$ and $H^* = \{ a \in H_0 \colon p \nmid a \ \text{for all} \ p \in P_0\}$, then $H_0 = \mathcal F (P_0) \times H^*$ (\cite[Theorem 1.2.3]{Ge-HK06a}). Clearly, $H^*$ is a reduced Krull monoid. By definition, $H$ is length-factorial if and only if $H_{\red} \cong H_0$ is length-factorial, and $H_0$ is length-factorial if and only if $H^*$ is length-factorial.

Let $H^* \hookrightarrow \mathcal F (P^*)$ be a divisor theory. Then $H_0  = \mathcal F (P_0) \times H^* \hookrightarrow \mathcal F (P_0) \times \mathcal F (P^*) = \mathcal F (P)$, where $P = P_0 \uplus P^*$,  is a divisor theory, whence we obtain that (we use \eqref{classgroup})
\[
\begin{aligned}
\mathcal C (H) & = \mathcal C (H_0) = \mathsf q ( \mathcal F (P))/\mathsf q (H_0) \\
               & =  \mathsf q ( \mathcal F (P_0)) \times \mathsf q ( \mathcal F (P^*)) / \mathsf q ( \mathcal F (P_0)) \times \mathsf q (H^*) \\
               & \cong \mathsf q ( \mathcal F (P^*)) / \mathsf q (H^*) = \mathcal C (H^*) \,.
\end{aligned}
\]
Let $G_{P^*} \subset \mathcal C (H^*)$ denote the set of classes containing prime divisors, and note that $0\not\in G_{P^*}$.
It remains to prove the characterization of length-factoriality.
Note that the {\it Moreover} statement, dealing with the case of torsion class groups, follows immediately from the main statement. We proceed in two steps.

\medskip
\noindent
{\bf Step 1.} Suppose that $H$ and $H^*$ are length-factorial but not factorial.

Assume to the contrary that there exist distinct $p, q \in P^*$  such that $0 \ne [p] = [q] \in \mathcal C (H^*)$. Since $H^* \hookrightarrow \mathcal F (P^*)$ is a divisor theory, there exist $r\ge 2$ and pairwise distinct $a_1,\ldots, a_r\in H^*$ such that $p=\gcd(a_1,\ldots, a_r)$. Without loss of generality, we may assume that $a_1,\ldots, a_r\in \mathcal A(H^*)$.

Let $a_1=p^kq_1 \cdot \ldots \cdot q_sp_2 \cdot \ldots \cdot p_{\ell}$, where $k\ge 1$, $s\ge 0$, $\ell\ge 1$, $q_1,\ldots, q_s\in P^* \setminus\{p\}$ with $[q_j]=[p]$ for $j\in [1,s]$, and $p_2,\ldots, p_{\ell}\in P^*$ with $[p_i]\neq [p]$ for $i\in [2,\ell]$. If $k+s\ge 2$, then $b_1=p^{k+s}p_2 \cdot \ldots \cdot p_{\ell}$ and $b_2=q^{k+s}p_2 \cdot \ldots \cdot p_{\ell}$ are both atoms of $H^*$. We observe that
\[
b_1 b_2 = (p^{k+s-1}q p_2 \cdot \ldots \cdot p_{\ell}) (p q^{k+s-1} p_2 \cdot \ldots \cdot p_{\ell})
\]
has two distinct factorizations of length two, a contradiction. Thus $k+s=1$ and $a_1=pp_2 \cdot \ldots \cdot p_{\ell}$. Similarly, we may assume that $a_2=pp_2' \cdot \ldots \cdot p_{\ell'}'$, where $\ell'\ge 2$ with $[p_i']\neq [p]$ for $i\in [2,\ell']$.
We observe that
\[
a_1 ( q p_2' \cdot \ldots \cdot p_{\ell'}') = (q p_2 \cdot \ldots \cdot p_{\ell}) a_2
\]
has two distinct factorizations of length two, a contradiction. Therefore, every nonzero class $g \in \mathcal C (H^*)$ contains at most one prime divisor. Thus, Lemma \ref{3.1}.2 implies that ${\boldsymbol \beta} \colon H^* \to \mathcal B (G_{P^*})$ is an isomorphism, whence $H^* \cong \mathcal B (G_{P^*})$
and $\mathcal B(G_{P^*})$ is length-factorial but not factorial.

It remains to determine the structure of $G_{P^*}$. Since $H^* \hookrightarrow \mathcal F (P^*)$ is a divisor theory and every class of $\mathcal C (H^*)$ contains at most one prime divisor, we obtain that $\mathcal C (H^*) = [G_{P^*} \setminus \{g\}]$ for all $g \in G_{P^*}$ by Lemma \ref{3.1}.3. Thus, the assumption of  Lemma \ref{3.3} is satisfied which implies that   $G_{P^*}$ is  a minimal non-half-factorial set. Let $B\in \mathcal B(G_{P^*})$ with $|\mathsf L(B)|\ge 2$ and let   $|B|$ be minimal with this property, say
\[
B=U_0U_1 \cdot \ldots \cdot U_k=V_0V_1 \cdot \ldots \cdot V_{\ell} \,,
\]
where $k,\ell\in \N$ with $k\neq \ell$, and $U_0,U_1,\ldots, U_k,V_0,V_1,\ldots, V_{\ell}\in \mathcal A(G_{P^*})$. Then Lemma \ref{3.2}.2 implies that
$$\{U_0,U_1,\ldots, U_k,V_0,V_1,\ldots, V_{\ell}\}=\mathcal A(G_{P^*})\,.$$
The minimality of $|B|$ implies that $U_i\neq V_j$ for every $i\in [0,k]$ and every $j\in [0,\ell]$.
If there exist $j\in [0,\ell]$ and a proper subset $I\subsetneq [0,k]$ such that $V_j$ divides $\prod_{i\in I}U_i$, then $\prod_{i\in I}U_i$ has two distinct factorizations, a contradiction to either the minimality of $|B|$ or the length-factoriality of $\mathcal B(G_{P^*})$. Therefore, $\gcd(U_i,V_j) \ne 1 \in \mathcal F (G_{P^*})$  for every $i\in [0,k]$ and $j\in [0,\ell]$, whence $|\gcd(U_i,V_j)|=1$ by Lemma \ref{3.3}.3. It follows that  $|U_i|=\ell+1$ and $|V_j|=k+1$ for every $i\in [0,k]$ and $j\in [0,\ell]$. Since
$$|\gcd(\prod_{i\in I}U_i, B)|=|I|(\ell+1),\ |\gcd(\prod_{i\in I}U_i,\prod_{j\in J}V_j)|\le |I||J|,  \text{ and } |\gcd(\prod_{i\in I}U_i,\prod_{j\in [0,\ell]\setminus J}V_j)|\le |I|(\ell+1-|J|)$$
for every $I\subset [0,k]$ and every $J\subset [0,\ell]$, we obtain that \begin{equation}\label{e1}
|\gcd(\prod_{i\in I}U_i,\prod_{j\in J}V_j)|=|I||J|\,.
\end{equation}
For every $g\in G_{P^*}$, there exist $i\in [0,k]$ and $j\in [0,\ell]$ such that $g\in \supp(U_i)\cap \supp(U_j)$. Then, by Lemma \ref{3.3}.2,  for any $i_1,i_2\in [0,k]$ and any $j_1,j_2\in [0,\ell]$
we have either $U_{i_1}=U_{i_2}$ or $\supp(U_{i_1})\cap \supp(U_{i_2})=\emptyset$
and either $V_{j_1}=V_{j_2}$ or $\supp(V_{j_1})\cap \supp(V_{j_2})=\emptyset$.

Assume to the contrary that there exist distinct $i_1,i_2\in [0,k]$ and distinct $j_1,j_2\in [0,\ell]$  such that $U_{i_1}=U_{i_2}$ and $V_{j_1}=V_{j_2}$.  Then $\gcd(U_{i_1}, V_{j_1})=g$ for some $g\in G_{P^*}$ and hence $\gcd(U_{i_1}U_{i_2}, V_{j_1}V_{j_2})=g^2$, a contradiction to Equation \eqref{e1}.
Thus, by symmetry, we may suppose $U_{i_1} \neq U_{i_2}$ for any distinct $i_1,i_2\in [0,k]$. Therefore $\supp(U_{i_1})\cap \supp(U_{i_2})=\emptyset$ for all distinct $i_1, i_2 \in [0,k]$.
Assume to the contrary that there exist $g\in G_{P^*}$ and $j\in [0,\ell]$ such that $\mathsf v_g(V_j)\ge 2$. Then there is $i\in [0,k]$ such that $\mathsf v_g(U_i)\ge 2$, and hence there is no atom $A
\in \mathcal A(G_{P^*})$ with $\mathsf v_g(A)=1$, a contradiction to Lemma \ref{3.3}.2. Thus $\mathsf v_g (V_j)=1$ for all $g \in \supp (V_j)$ and all  $j\in [0,\ell]$.

We set $U_1=g_1^{s_0}e_{1,1}^{s_1} \cdot \ldots \cdot e_{1,t}^{s_t}$, where $s_0,\ldots,s_t\in \N$ and $g_1,e_{1,1},\ldots, e_{1,t}\in G_{P^*}$ are pairwise distinct. After renumbering if necessary, we may suppose $e_{1,i}\in \supp(V_i)$ for every $i\in [1,t]$ and $g_1 \in \supp(V_0)$.
Note that if $\supp(V_{j_1})\cap \supp(V_{j_2})\neq \emptyset$, then $V_{j_1}=V_{j_2}$, where $j_1,j_2\in [0,\ell]$.
Therefore,
\[
B=U_0 \cdot \ldots \cdot U_k=V_0 \cdot \ldots \cdot V_{\ell}=V_0^{s_0}V_1^{s_1} \cdot \ldots \cdot V_t^{s_t}\,.
\]
The length-factoriality of $\mathcal B(G_{P^*})$ implies that $k+1\neq s_0+\ldots+s_t$.
Since $\supp(U_{i_1})\cap \supp(U_{i_2})=\emptyset$ for any two distinct $i_1,i_2\in [0,k]$,
 $U_1, \ldots, U_k$ and $V_1, \ldots , V_t$ can be written as the form
\[
U_i=g_i^{s_0}e_{i,1}^{s_1} \cdot \ldots \cdot e_{i,t}^{s_t}, \ V_j=e_{0,j}e_{1,j} \cdot \ldots \cdot e_{k,j}
\]
and
\[
U_0=g_0^{s_0}e_{0,1}^{s_1}\ldots e_{0,t}^{s_t}, \quad \text{ and } \quad V_{0}=g_0g_1 \cdot \ldots \cdot g_k \,,
\]
where $e_{2,1},\ldots, e_{2,t},\ldots, e_{k,1},\ldots,e_{k,t}, g_2,\ldots,g_k\in G_{P^*}$,
$g_0=-\sum_{i=1}^kg_i$, and $e_{0, j}=-\sum_{i=1}^ke_{i,j}$ for every $j\in [1,t]$.
For each $i\in [0,k]$,   $\mathcal B(\{g_i,e_{i,1},\ldots,e_{i,t}\})$ is half-factorial and length-factorial, whence it is factorial and $\mathcal A(\{g_i,e_{i,1},\ldots,e_{i,t}\})=\{U_i\}$.  Thus, we obtain that $s_0$ is the minimal integer such that $-s_0g_i\in \langle e_{i,1},\ldots,e_{i,t}\rangle$.

In order to  show that $(e_{1,1},\ldots, e_{1,t}, e_{2,1},\ldots, e_{2,t}, \ldots, e_{k-1,1},\ldots, e_{k-1,t})$ is  independent we set
\[
G_1=\{e_{1,1},\ldots, e_{1,t}, e_{2,1},\ldots, e_{2,t}, \ldots, e_{k-1,1},\ldots, e_{k-1,t}\} \,.
\]
Assume to the contrary that the above tuple is not independent. Then there are two distinct  $T_1, T_2 \in \mathcal F (G_1)$ such that $\sigma(T_1)=\sigma(T_2)$. By symmetry, we may assume that $T_1 \ne 1_{\mathcal F (G_1)}$.
There exist non-negative integers $x_1,\ldots, x_t$ with $x_1+\ldots+x_t=|T_1|$ such that $T_1$ divides $V_1^{x_1} \cdot \ldots \cdot V_t^{x_t}$ in $\mathcal F (G_1)$, whence $V_1^{x_1} \cdot \ldots \cdot V_t^{x_t}T_2T_1^{-1}$ is a zero-sum sequence. Since $V_1^{x_1} \cdot \ldots \cdot V_t^{x_t}T_2T_1^{-1}$ has only one factorization and $V_1,\ldots, V_t$ are the only atoms dividing $V_1^{x_1} \cdot \ldots \cdot V_t^{x_t}T_2T_1^{-1}$, it follows that $V_1^{x_1} \cdot \ldots \cdot V_t^{x_t}T_2T_1^{-1}=V_1^{x_1} \cdot \ldots \cdot V_t^{x_t}$ and hence $T_1=T_2$, a contradiction.

 Next we show that $\langle g_i,e_{i,1},\ldots, e_{i,t}\rangle \cap \langle g_j,e_{j,1},\ldots, e_{j,t}\colon j\in [1,k]\setminus\{i\} \rangle =\{0\}$ for every $i\in [1,k]$. Assume to the contrary that there exists $0\neq h\in \langle g_i,e_{i,1},\ldots, e_{i,t}\rangle \cap \langle g_j,e_{j,1},\ldots, e_{j,t}\colon j\in [1,k]\setminus\{i\} \rangle$. Since $\langle g_i,e_{i,1},\ldots, e_{i,t}\rangle=[ g_i,e_{i,1},\ldots, e_{i,t}]$ and $\langle g_j,e_{j,1},\ldots, e_{j,t}\colon j\in [1,k]\setminus\{i\} \rangle=[ g_j,e_{j,1},\ldots, e_{j,t}\colon j\in [1,k]\setminus\{i\} ]$, there exist a zero-sum free sequence $T_1$ over $\{g_i,e_{i,1},\ldots, e_{i,t}\}$ and a zero-sum free sequence $T_2$ over $\{ g_j,e_{j,1},\ldots, e_{j,t}\colon j\in [1,k]\setminus\{i\} \}$ such that $h=\sigma(T_1)=\sigma(T_2)$. Let $N$ be large enough such that $T_1$ divides $U_i^N$. Then $U_i^NT_2T_1^{-1}$ is a zero-sum sequence such that $\supp(U_i^NT_2T_1^{-1})\cap \{ g_j,e_{j,1},\ldots, e_{j,t}\colon j\in [1,k]\setminus\{i\} \}\neq\emptyset$, which implies that there exists $\nu \in [1,k]\setminus\{i\}$ such that $U_{\nu}$ divides $U_i^NT_2T_1^{-1}$ and hence $U_{\nu}$ divides $T_2$, a contradiction.
 Therefore, we obtain that
 \[
 \mathcal C (H^*) =\langle G_{P^*}\rangle=\langle e_{1,1},\ldots,e_{1,t},g_1 \rangle \oplus \ldots \oplus \langle e_{k,1},\ldots,e_{k,t},g_k \rangle\,.
\]

Let $i\in [1,k]$ and set $G_i=\langle g_i,e_{i,1},\ldots,e_{i,t} \rangle$. Then  $G_i \cong\Z^t\oplus \Z/m\Z$, where $m\in\N$ is the maximal order of all the torsion elements of $G_i$. Let $\gcd(s_0,s_1,\ldots,s_t)=n$. Then the fact that $h=\sigma(g_i^{s_0/n}e_{i,1}^{s_1/n} \cdot \ldots \cdot e_{i,t}^{s_t/n})$ has order $n$ implies that $n\le m$. It remains to verify that $n\ge m$. Let $\alpha\in G_i$ such that $\ord(\alpha)=m$.
Suppose $\alpha=w_0g_i+w_1e_{i,1}+\ldots+w_te_{i,t}$, where $w_0,\ldots,w_t\in \N_0$. Then
$(g_i^{w_0}e_{i,1}^{w_1} \cdot \ldots \cdot e_{i,t}^{w_t})^{m}=U_i^w$ for some $w\in \N$ with $\gcd(m,w)=1$, which implies that $m$ divides $\gcd(s_0,s_1,\ldots,s_t)=n$.

\smallskip
\noindent
{\bf Step 2.}  Suppose that $H^* \cong \mathcal B (G_{P^*})$ and that $G_{P^*}$ has the given form. We have to show that $\mathcal B (G_{P^*})$ is length-factorial but not factorial.

We use the simple fact that if an abelian group $G$ is a direct sum, say $G = G_1 \oplus G_2$, and if $G_i' \subset G_i$ are subsets for $i \in [1,2]$, then $\mathcal A (G_1' \uplus G_2') = \mathcal A (G_1') \uplus \mathcal A (G_2')$.
We define, for every $i\in [0,k]$ and  every $j\in [1,t]$,
\[
U_0 = g_0^{s_0}e_{0,1}^{s_1} \cdot \ldots \cdot e_{0,t}^{s_t}, \ U_i=e_{i,0}^{s_0} \cdot \ldots \cdot e_{i,t}^{s_t}, \ V_0=g_0 \cdot \ldots \cdot  g_k, \quad \text{ and } \quad V_j=e_{0,j} \cdot \ldots \cdot  e_{k,j} \,.
\]
Clearly, we obtain that
\begin{equation} \label{structure of atoms}
\mathcal A(G_{P^*})=\{U_0,\ldots, U_k, V_0,\ldots, V_t\} \quad \text{and} \quad U_0 \cdot \ldots \cdot U_k=V_0^{s_0} \cdot \ldots \cdot V_t^{s_t} \,.
\end{equation}
Thus, $\mathcal B (G_{P^*})$ is not factorial.
By definition, we have $|U_i|=\sum_{j=0}^ts_j$, $|V_j|=k+1$ for every $i\in [0,k]$ and every $j\in [0,t]$.
Assume to the contrary that there exists $B_0\in \mathcal B(G_{P^*})$ such that $B_0$ has two distinct factorizations of the same length. We may assume that $B_0$ is a counterexample with minimal length.
Suppose $$B_0=\prod_{i\in I_1}U_i^{a_i}\prod_{j\in J_1}V_j^{b_j} \ \text{ and } \ B_0=\prod_{i\in I_2}U_i^{a_i'}\prod_{j\in J_2}V_j^{b_j'}$$ are two distinct factorizations of the same length,
where $I_1,I_2\subset [0,k]$, $J_1,J_2\subset [0,t]
$, $a_i\in \N_0$ for every $i\in I_1$, $a_i'\in \N_0$ for every $i\in I_2$, $b_j\in \N_0$ for every $j\in J_1$, and $b_j'\in \N_0$ for every $j\in J_2$. The minimality of $|B_0|$ implies that $I_1\cap I_2=\emptyset$ and $J_1\cap J_2=\emptyset$. If $I_1\cup I_2=\emptyset$, then those two factorizations of $B_0$ must be equal, a contradiction. By symmetry, we may suppose
$I_1\neq\emptyset$. Then $\cup_{i\in I_1}\supp(U_i)\subset \supp(\prod_{i\in I_2}U_i^{a_i'}\prod_{j\in J_2}V_j^{b_j'})$ implies that $J_2=[0,t]$ and $J_1=\emptyset$, whence $I_1=[0,k]$ and $I_2=\emptyset$. It follows that
$$B_0=\prod_{i=0}^kU_i^{a_i}=\prod_{j=0}^tV_j^{b_j'}\,,$$
whence $(s_0+\ldots +s_t)\sum_{i=0}^ka_i=\sum_{i=0}^ka_i|U_i|=|B_0|=\sum_{j=1}^tb_j'|V_j|=(k+1)\sum_{j=1}^tb_j'$.
Since $s_0+\ldots +s_t\neq k+1$, we obtain $\sum_{i=0}^ka_i\neq \sum_{j=1}^tb_j'$, a contradiction to the fact that the two factorizations have the same length.
\end{proof}

\smallskip
The system of sets of lengths $\mathcal L (H)$ of an atomic monoid $H$ is said to be {\it additively closed} if the sumset $L_1 + L_2 \in \mathcal L (H)$ for all $L_1, L_2 \in \mathcal L (H)$. Clearly,  $L_1 + L_2 = L_1$ implies that $L_2 = \{0\}$ for all nonempty sets $L_1, L_2 \subset \N_0$, whence set addition is a  unit-cancellative operation. Thus, $\mathcal L (H)$ is additively closed if and only if $(\mathcal L (H), + )$ is a reduced monoid with set addition as operation.

Let $H$ be a  Krull monoid class group $G$ and let $G_0 \subset G$ denote the set of classes containing prime divisors. Then the inclusion $\mathcal B ( G_0) \hookrightarrow \mathcal F (G_0)$ is a divisor homomorphism but it need not be a divisor theory (\cite{Sc10a}). In Corollary \ref{1.2} we prove that in case of length-factorial Krull monoids this inclusion is a divisor theory.

\smallskip
\begin{proof}[Proof of Corollary \ref{1.2}]
Let $H$ be a length-factorial Krull monoid, that is not factorial,  and let all notation be as in Theorem \ref{1.1}.

1. Since the inclusion $H^* \hookrightarrow \mathcal F (P^*)$ is a divisor theory, $H^* \cong \mathcal B (G_{P^*})$, and every class of $G_{P^*}$ contains precisely one prime divisor, the inclusion $\mathcal B (G_{P^*}) \hookrightarrow \mathcal F (G_{P^*})$ is a divisor theory with class group isomorphic to $\mathcal C (H^*) \cong \mathcal C (H)$. The assertion on $\mathcal A (G_{P^*})$ follows from Equation \eqref{structure of atoms}.

\smallskip
2. Let $B \in \mathcal B (G_{P^*})$ and $z \in \mathsf Z (B)$. By \eqref{structure of atoms},
$z$  can be written in the form
\[
z =\prod_{i=0}^k U_i^{c_i}\prod_{j=0}^{t}V_j^{d_j} \in \mathsf Z (B) \,,
\]
where $c_i,d_j\in \N_0$ for every $i\in [0,k]$ and every $j\in [0,t]$, and we have to determine the relations between the exponents $c_1, \ldots, c_k, d_1, \ldots , d_t$.
Let $$x_1=\min\{c_i\colon i\in [0,k]\} \text{ and } x_2=\min\left\{\left\lfloor\frac{d_j}{s_j}\right\rfloor\colon j\in [0,t]\right\}\,.$$
Then
\[
z=\prod_{i=0}^k U_i^{c_i}\prod_{j=0}^{t}V_j^{d_j}=(U_0 \cdot \ldots \cdot U_k)^{x_1}(V_0^{s_0} \cdot \ldots \cdot V_t^{s_t})^{x_2}\prod_{i=0}^k U_i^{c_i-x_1}\prod_{j=0}^{t}V_j^{d_j-x_2s_j}\,.
\]
We set $x=x_1+x_2$, $y_i=c_i-x_1$, and $z_j=d_j-x_2s_j$ for every $i\in [0,k]$ and every $j\in [0,t]$.
Thus,
\[
B =(U_0 \cdot \ldots \cdot U_k)^x\ \prod_{i=0}^k U_i^{y_i} \ \prod_{j=0}^tV_j^{z_j}
\]
has a factorization of the required form. Since for every $\nu \in [0,x]$,
\[
z' =  (U_0 \cdot \ldots \cdot U_k)^{\nu} (V_0^{s_0} \cdot \ldots \cdot V_t^{s_t})^{x - \nu} \ \prod_{i=0}^k U_i^{y_i} \ \prod_{j=0}^tV_j^{z_j}    \in \mathsf Z (B) \,,
\]
we have $$|z|\in \sum_{i=0}^ky_i + \sum_{j=0}^t z_j + \Big\{\nu (k+1)+(x-\nu)\sum_{j=0}^ts_j\colon \nu \in [0,x] \Big\} \subset \mathsf L(B)\,.$$
If $B$ can be written uniquely in the asserted form then,
since $z$ is chosen arbitrary, it follows that
\[
\mathsf L (B) =  \sum_{i=0}^ky_i + \sum_{j=0}^t z_j + \Big\{\nu (k+1)+(x-\nu)\sum_{j=0}^ts_j\colon \nu \in [0,x] \Big\} \,.
\]
It remains to verify  the uniqueness assertion. Suppose that
\[
B=(U_0 \cdot \ldots \cdot U_k)^x\ \prod_{i=0}^kU_i^{y_i} \ \prod_{j=0}^tV_j^{z_j}=(U_0 \cdot \ldots \cdot U_k)^{x'}\ \prod_{i=0}^kU_i^{y_i'} \ \prod_{j=0}^tV_j^{z_j'}\,, \quad \text{where}
\]
\begin{itemize}
\item $x, y_0,\ldots, y_k, z_0,\ldots,z_t \in \N_0$,   $y_i=0$ for some  $i\in [0,k]$,  and   $z_j<s_j$ for some $j\in [0,t]$, and

\item     $x', y_0',\ldots, y_k', z_0',\ldots,z_t' \in \N_0$,   $y_{i_0}' =0$ for some  $i_0 \in [0,k]$,  and   $z_{j_0}' < s_{j_0} $ for some $j_0 \in [0,t]$.
\end{itemize}
Note, if there would exist $i\in [0,k]$ such that $U_i$  divides $\prod_{j=1}^tV_j^{z_j'}$, then $$s_{j_0}=\mathsf v_{e_{i,j_0}}(U_i)\le \mathsf v_{e_{i,j_0}}(\prod_{j=1}^tV_j^{z_j'})=z_{j_0'}\,,$$
a contradiction.
If $x>x'$, then $U_{i_0}^{y_{i_0}+1}$ divides $\prod_{i=1}^kU_i^{y_i'} \ \prod_{j=1}^tV_j^{z_j'}$. Since $\supp(U_{i_0})\cap \supp(U_i)=\emptyset$ for every $i\in [0,k]\setminus\{i_0\}$, we have $U_{i_0}^{y_{i_0}+1}$ divides $U_{i_0}^{y_{i_0}'} \ \prod_{j=1}^tV_j^{z_j'}=\prod_{j=1}^tV_j^{z_j'}$,
 a contradiction. Thus $x\le x'$. By symmetry, we obtain that  $x'\le x$, whence $x=x'$.
If $y_i>y_i'$ for some $i\in [0,k]$, then $U_i$ must divide $\prod_{j=1}^tV_j^{z_j'}$, a contradiction. Thus $y_i\le y_i'$ for every $i\in [0,k]$. By symmetry, we obtain  that $y_i'\le y_i$, whence $y_i=y_i'$ for every $i\in  [0,k]$.
Since $x=x'$ and $y_i=y_i'$ for every $i\in  [0,k]$, we infer that  $ \prod_{j=1}^tV_j^{z_j}=\prod_{j=1}^tV_j^{z_j'}$, whence $z_j=z_j'$ for every $j\in [0,t]$.

\smallskip
3. By Lemma \ref{2.3}.3, Lemma \ref{3.1}.1, and Theorem \ref{1.1}, we have $\mathcal L(H)=\mathcal L \big( \mathcal B (G_{P^*}) \big)$. By item 2., we infer  that
\[
\mathcal L(H)\subset \Big\{  \big\{y+ \nu (k+1)+(x-\nu)\sum_{j=0}^ts_j\colon \nu \in [0,x] \big\} \colon y,x\in \N_0 \Big\} \,.
\]
Conversely, if $y,x\in \N_0$ and  $B=(U_0 \cdot \ldots \cdot U_k)^xU_0^y$,
then $\{y+\nu (k+1)+(x-\nu )\sum_{j=0}^ts_j \colon \nu \in [0,x]\}=\mathsf L(B)\in \mathcal L(H)$, whence
\[
\mathcal L(H)= \Big\{  \big\{y+ \nu (k+1)+(x-\nu)\sum_{j=0}^ts_j\colon \nu \in [0,x] \big\} \colon y,x\in \N_0 \Big\} \,.
\]
The given description shows immediately that  $\mathcal L(H)$ is additively closed with respect to set addition.
\end{proof}

\smallskip
Before proving Corollary \ref{1.3} we briefly recall the involved concepts. Let $H$ be a Krull monoid and $H_{\red} \hookrightarrow F = \mathcal F (P)$ be a divisor theory. Then $H$ satisfies the {\it approximation property} if one of the following equivalent conditions is satisfied (\cite[Proposition 2.5.2]{Ge-HK06a}):
\begin{itemize}
\item[(a)] For all $n \in \N$ and distinct $p, p_1, \ldots, p_n \in P$ there exists some $a \in H$ such that $\mathsf v_p (a)=1$ and $\mathsf v_{p_i} (a) = 0$ for all $i \in [1,n]$.

\item[(b)] For all $a, b\in F$, there exists some $c \in F$ such that $[a]=[c] \in G$ and $\gcd (b,c)=1$.
\end{itemize}

\begin{proof}[Proof of Corollary \ref{1.3}]
Let $H$ be a Krull monoid. Without restriction we may suppose that $H$ is reduced. Using the notation of Theorem \ref{1.1}, we have $H = \mathcal F (P_0) \times H^*$ and a divisor theory $\mathcal F (P_0) \times H^* \hookrightarrow \mathcal F (P_0) \times \mathcal F (P^*)$. Let $G_{P^*} \subset \mathcal C (H^*) \cong \mathcal C (H)$ denote the set of classes containing prime divisors.

1. If $H$ and $H^*$ are length-factorial but not factorial, then $P^*$ is finite by Theorem \ref{1.1}. Thus, Condition (b) above cannot hold, whence $H$ does not satisfy the approximation property.

\smallskip
2. Suppose that every nonzero class of $G = \mathcal C (H^*)$ contains a prime divisor. Note that $0 \in \mathcal A ( G)$ is the only prime element of $\mathcal B (G)$ and  $\mathcal B (G) = \mathcal F ( \{0\}) \times \mathcal B (G \setminus \{0\})$. Thus $\mathcal B (G)$ is length-factorial if and only if $\mathcal B (G \setminus \{0\})$ is length-factorial.

First, we suppose that $H^* \cong \mathcal B (G \setminus \{0\})$ and that either $|G| \le 3$ or $G \cong C_2 \oplus C_2$. We have to verify that $\mathcal B (G)$ is length-factorial.
If $|G| \le 2$, then $\mathcal B (G)$ is factorial and hence length-factorial. If $|G|=3$ or $G \cong C_2 \oplus C_2$, then it can be checked directly that $\mathcal B (G)$ is length-factorial.

Conversely, suppose that $H^*$ is length-factorial. Since $G \setminus \{0\} \subset G_{P^*}$, the description of $G_{P^*}$ achieved in Theorem \ref{1.1} implies that $|G| \le 3$ or $G \cong C_2 \oplus C_2$.
\end{proof}

In order to prove Corollary \ref{1.4}, we first gather some basics from the theory of rings with zero-divisors. Let $R$ be a commutative ring with identity and let $R^{\bullet}$ denote its monoid of regular elements. Then $R$ is {\it additively regular} if for each pair of elements $a,b \in R$ with $b$ regular, there is an element $r \in R$ such that $a+br$ is a regular element of $R$ (\cite{Hu88a,Lu16a}). Every additively regular ring is a Marot ring and every Marot ring is a $v$-Marot ring. The ring $R$ is a Krull ring if it is completely integrally closed and satisfies the ACC on regular divisorial ideals. If $R$ is a Krull ring, then $R^{\bullet}$ is a Krull monoid and if $R$ is a $v$-Marot ring, then the converse holds (\cite[Theorem 3.5]{Ge-Ra-Re15c}). We say that $R$ is atomic (factorial, half-factorial, resp. length-factorial) if $R^{\bullet}$ has the respective property.

Next we need the concept of normalizing Krull rings. A cancellative but not necessarily commutative semigroup $S$ (resp. a ring $R$) is said to be {\it normalizing} if $aS = Sa$ for all $a \in S$ (resp. $aR = Ra$ for all $a \in R$). A prime Goldie ring is said to be a Krull ring (or a Krull order) if it is completely integrally closed (equivalently, a maximal order) and satisfies the ACC on two-sided divisorial ideals. Thus, every commutative Krull domain is a normalizing Krull ring. For examples and background on  non-commutative (normalizing) Krull rings we refer to \cite{Wa-Je86a, Le-Oy86, Ok16a, Ak-Ma16a}, and for background on factorizations in the non-commutative setting  to \cite{Ba-Sm15, Sm16a}. In particular, normalizing Krull monoids are transfer Krull.

\smallskip
\begin{proof}[Proof of Corollary \ref{1.4}]
1. Let $R$ be an additively regular Krull ring. Then $R^{\bullet}$ satisfies the approximation property by \cite[Theorem 2.2]{Os99a} (this  needs the assumption that $R$ is additively regular). Thus $R^{\bullet}$ is a Krull monoid satisfying the  approximation property, whence the assertion follows from Corollary \ref{1.3}.1.

2. Let $R$ be a normalizing Krull ring. Then $R$ satisfies the approximation property (\cite[Proposition 2.9]{Le-Oy86}, \cite[Theorem 4]{Ma82a}). If $H$ denotes the monoid of regular elements, then $H_{\red}$ is a commutative Krull monoid by \cite[Corollary 4.14 and Proposition 5.1]{Ge13a}. Thus, the assertion follows from Corollary \ref{1.3}.1.
\end{proof}

\smallskip
\begin{proof}[Proof of Corollary \ref{1.5}]
Let $H$ be a length-factorial transfer Krull monoid. We have to show that $H_{\red}$ is a Krull monoid. Since $H_{\red}$ is a length-factorial transfer Krull monoid, we may suppose that $H$ is reduced. Let $B$ be a Krull monoid and let $\theta' \colon H \to B$ be a transfer homomorphism. We may suppose that $B$ is reduced and start with the following assertion.

\noindent
{\bf A.} $H$ is cancellative.

\noindent
{\it Proof of} \,{\bf A}.\,
Let $a,b,c\in H$  such that $ab=ac$.   Since $\theta'(a)\theta'(b)=\theta'(a)\theta'(c)$, we obtain that $\theta'(b)=\theta'(c)$. If $\theta'(b)=\theta'(c)=1_B$, then $b=c=1_H$. If $\theta'(b)=\theta'(c)= w_1 \cdot \ldots \cdot w_r$, where $r\in \N$ and $w_1,\ldots,w_r\in \mathcal A(B)$, then there exist $b_1,\ldots,b_r,c_1,\ldots,c_r\in \mathcal A(H)$ such that
	$b=b_1 \cdot \ldots \cdot b_r$ and $c=c_1 \cdot \ldots \cdot c_r$. Suppose $a=a_1 \cdot \ldots \cdot a_k$, where $k\in \N_0$ and $a_1,\ldots, a_k\in \mathcal A(H)$. Then the two factorizations
	$z_1=a_1 \cdot \ldots \cdot a_kb_1 \cdot \ldots \cdot b_r \in \mathsf Z (ab)$ and $z_2=a_1 \cdot \ldots \cdot a_kc_1 \cdot \ldots \cdot c_r \in \mathsf Z (ab)$ of $ab$ have the same length $k+r$, whence  $z_1=z_2$. Thus $b_1 \cdot \ldots \cdot b_r=c_1 \cdot \ldots \cdot c_r \in \mathsf Z (H)$, whence $b=c \in H$.	{\it  \qed (Proof of {\bf A}).}

\smallskip
Thus, $H$ is a reduced cancellative length-factorial transfer Krull monoid. If $H$ is factorial, then $H$ is Krull.
Suppose that $H$ is  not factorial. Then $H$ is not half-factorial. Let $G$ be the class group of $B$ and let $G_0 \subset G$ be the set of classes containing prime divisors. Since $H$ is not factorial, it is not half-factorial. Thus, Lemma \ref{2.3} implies that $B$ is length-factorial but not half-factorial. Theorem \ref{1.1} implies that every class of $G_0$ contains precisely one prime divisor.
Lemma \ref{3.1}.1 implies that there is a transfer homomorphism $\boldsymbol \beta \colon B \to \mathcal B (G_0)$. Since every class of $G_0$ contains precisely one prime divisor, Lemma \ref{3.1}.3 implies that  $G=[G_0\setminus\{g\}]$ for every $g\in G_0$. Since the composition of transfer homomorphisms is a transfer homomorphism again, we obtain a transfer homomorphism $\theta = \boldsymbol \beta \circ \theta' \colon H \to \mathcal B (G_0)$.

Let $P_0 \subset H$ be the set of prime elements of $H$ and $H_0 = \{ a \in H \colon p \nmid a \ \text{for all} \ p \in P_0 \}$. Since $H$ is cancellative, we obtain that $H = \mathcal F (P_0) \times H_0$.
 Since $G=[G_0\setminus\{g\}]$ for every $g\in G_0$,
 the only possible prime element of $\mathcal B(G_0)$ is the  sequence $S = 0 \in \mathcal F (G_0)$. Thus Lemma \ref{2.3}.2 implies that, if $P_0 \neq \emptyset$, then $\theta(P)=\{0\}$.
  Thus, we obtain that $\theta(H_0)=\mathcal B(G_0\setminus\{0\})$ and hence $\theta_{H_0}\colon H_0\rightarrow \mathcal B(G_0\setminus\{0\})$ is a surjective transfer homomorphism.
By Lemma \ref{2.3},  $\mathcal B(G_0\setminus\{0\})$ is length-factorial but  not half-factorial. By Corollary \ref{1.2}.1, $\mathcal A(G_0\setminus\{0\})$ is finite, say $\mathcal A(G_0\setminus\{0\})=\{U_1',\ldots, U_k', V_1',\ldots, V_{\ell}'\}$, $U_1' \cdot \ldots \cdot U_k'=V_1' \cdot \ldots \cdot V_{\ell}'$,  $k, \ell \in \N_{\ge 2}$, $k \ne \ell$, and  $U_i' \ne V_j'$ for all $i \in [1,k]$ and $j \in [1, \ell]$.

Assume to the contrary that $\theta_{H_0}$ is not injective. Then there exist $a,b\in H_0$ with $a\neq b$ such that $T =\theta(a)=\theta(b)$, say $T=W_1 \cdot \ldots \cdot W_r$, where $r\in \N$ and $W_1,\ldots, W_r\in \mathcal A(G_0\setminus\{0\})$. Then there exist $a_1,\ldots, a_r,b_1,\ldots,b_r\in \mathcal A(H_0)$ such that $a=a_1 \cdot \ldots \cdot  a_r$, $b=b_1 \cdot \ldots \cdot  b_r$, and $\theta(a_i)=\theta(b_i)=W_i$  for all $i\in [1,r]$. Since $a\neq b$, there exists $i_0\in [1,r]$, say $i_0=1$, such that $a_1\neq b_1$. After renumbering if necessary, we may suppose $W_1=U_1'$. Let $c\in H_0$ such that $\theta(c)=\prod_{i=2}^kU_i'$. Therefore, $\theta(a_1c)=\theta(b_1c)=U_1' \cdot \ldots \cdot  U_k'=V_1' \cdot \ldots \cdot  V_{\ell}'$, which implies that there exist $u_1,\ldots, u_{\ell}, v_1,\ldots,v_{\ell}$ such that $a_1c=u_1 \cdot \ldots \cdot  u_{\ell}$, $b_1c=v_1 \cdot \ldots \cdot v_{\ell}$, and $\theta(u_j)=\theta(v_j)=V_j'$ for all $j \in [1, \ell]$.
We observe that
$$a_1b_1c=a_1v_1 \cdot \ldots \cdot v_{\ell}=b_1u_1 \cdot \ldots \cdot u_{\ell}\,.$$
If there exists $j\in [1,\ell]$ such that $a_1=u_j$, then $U_1'=\theta(a_1)=\theta(u_j)=V_j'$, a contradiction. Thus $a_1b_1c$ has two distinct factorization of length $\ell+1$, a contradiction.
Therefore $\theta_{H_0}$ is injective, whence $H_0\cong \mathcal B(G_0\setminus\{0\})$ is Krull and so $H=\mathcal F(P_0)\times H_0$ is Krull.
\end{proof}

The monoids, discussed in Example \ref{2.2}.2, are reduced and length-factorial but not cancellative. Thus they cannot be transfer Krull by Corollary \ref{1.5}. We end with an example of transfer Krull monoids.

\smallskip
\begin{example} \label{2.7}
Let $R$ be a Bass ring and let $\mathsf T (R)$ be the monoid of isomorphism classes of torsion-free finitely generated $R$-modules, together with the operation induced by the direct sum (this is a monoid as discussed in Example \ref{2.2}.1). Then $\mathsf T (R)$ is a reduced transfer Krull monoid by \cite[Theorem 1.1]{Ba-Sm21a}. There are algebraic characterizations of when $\mathsf T (R)$ is
{factorial,  resp. half-factorial,   resp. cancellative} (see \cite[Proposition 3.13, Corollary 1.2, Remark 3.17]{Ba-Sm21a}). These characterizations show that $\mathsf T (R)$  is rarely cancellative, whence rarely Krull, and thus, by Corollary \ref{1.3}, it is rarely length-factorial.
\end{example}

\medskip
\noindent
{\bf Acknowledgement}. We thank the reviewers for their careful reading and for all their comments which led to the introduction of length-FF monoids.

\bigskip

\providecommand{\bysame}{\leavevmode\hbox to3em{\hrulefill}\thinspace}
\providecommand{\MR}{\relax\ifhmode\unskip\space\fi MR }
\providecommand{\MRhref}[2]{%
  \href{http://www.ams.org/mathscinet-getitem?mr=#1}{#2}
}
\providecommand{\href}[2]{#2}

\end{document}